\documentclass[10pt,oneside]{amsart}

\usepackage[margin=1in]{geometry}
\usepackage{microtype}
\usepackage{savesym}
\usepackage{amsmath}
\usepackage{amsthm}
\usepackage{amssymb} %gives you all the symbols you could want
\usepackage{bbm} %more extensive blackboard bold alphabet
\usepackage{mathabx} %more symbols
\usepackage{hyperref}
\usepackage{cleveref}
\usepackage{array}
\usepackage{enumitem}
\usepackage[nocompress]{cite}
\usepackage[pdftex]{graphicx}
\usepackage[caption=false,font=footnotesize,labelfont=sf,textfont=sf]{subfig}
\usepackage{url}
\usepackage{xcolor}
\usepackage[draft, inline,nomargin]{fixme}
\let \P \relax

\newcommand{\A}{\mathbb{A}}
\newcommand{\P}{\mathbb{P}}

\newcommand{\Z}{\mathbb{Z}}

\newcommand{\R}{\mathbb{R}}
\newcommand{\C}{\mathbb{C}}

\newcommand{\la}{\lambda}

\newcommand{\si}{\sigma}

\newcommand{\minus}{\smallsetminus}
\renewcommand{\bar}{\overline}
\renewcommand{\tilde}{\widetilde}
\renewcommand{\phi}{\varphi}

\let \< \relax
\let \> \relax
\newcommand{\<}{\langle}
\newcommand{\>}{\rangle}

\newcommand{\mn}[1]{{\left\vert\kern-0.25ex\left\vert\kern-0.25ex\left\vert #1 \right\vert\kern-0.25ex\right\vert\kern-0.25ex\right\vert}}

\newcommand{\BM}{\begin{matrix}}
\newcommand{\EM}{\end{matrix}}

\DeclareMathOperator{\GL}{GL}

\DeclareMathOperator{\Span}{Span}

\DeclareMathOperator{\diag}{diag}

\newtheorem{theorem}{Theorem}[section]
\newtheorem{definition}[theorem]{Definition}

\newtheorem{lemma}[theorem]{Lemma}
\newtheorem{example}[theorem]{Example}
\newtheorem{corollary}[theorem]{Corollary}

\newcommand{\etal}{\textit{et al}.~}
\newcommand{\ie}{\textit{i}.\textit{e}.~}

\begin{document}
\title{Ideals of the Multiview Variety}
\author{Sameer Agarwal and Andrew Pryhuber and Rekha R. Thomas}
\thanks{Pryhuber and Thomas were partially supported by the NSF grant DMS-1719538}
\date{\today}

\begin{abstract}
The multiview variety of an arrangement of cameras is the Zariski closure of the images of world points in the cameras. The prime vanishing ideal of this complex projective variety is called the multiview ideal. We show that the bifocal and trifocal polynomials from the cameras generate the multiview ideal when the foci are distinct. In the computer vision literature, many sets of (determinantal) polynomials have been proposed to describe the multiview variety. 
We establish precise algebraic relationships between the multiview ideal and these various ideals. When the camera foci are noncoplanar, we prove that the ideal of bifocal polynomials saturate to give the multiview ideal. Finally, we prove that all the ideals we consider coincide when dehomogenized, to cut out the space of finite images. 
\end{abstract}

% \begin{IEEEkeywords}
% %???
% %Computer Society, IEEE, IEEEtran, journal, \LaTeX, paper, template.
% \end{IEEEkeywords}
%}

\maketitle

\section{Introduction}\label{sec:introduction} 
A general projective camera is a rank three matrix in $\mathbb{R}^{3 \times 4}$. Given a camera arrangement $\mathcal{A} = (A_1, \ldots, A_n)$, 
the image formation map 
$$\varphi_\mathcal{A} \,:\, \mathbb{P}_\R^3 \dashrightarrow (\mathbb{P}_\R^2)^n$$ sends a homogenized world point $\mathbf{q} \in \mathbb{P}_\R^3$ to its images $(\mathbf{p}_1 = A_1 \mathbf{q}, \ldots, \mathbf{p}_n = A_n \mathbf{q}) \in (\P_\R^2)^n$. The $i$th copy of $\P_\R^2$ in the codomain of $\varphi_\mathcal{A}$  is the homogenized image plane of camera $i$.  The unique point $\mathbf{c}_i \in \P_\R^3$ in the kernel of $A_i$ is the focal point of camera $i$. The map $\varphi_\mathcal{A}$ is defined at all points in $\P_\R^3$ except at the foci $\mathbf{c}_1, \ldots, \mathbf{c}_n$. 
Triggs called $\varphi_\mathcal{A}(\P_\R^3)$ the {\em joint image}~\cite{triggs95} and Heyden-\AA str\"{o}m call it the {\em natural descriptor}~\cite{HA97}. We are interested in studying the complete 
set of polynomials that vanish on $\varphi_\mathcal{A}(\P_\R^3)$.

\begin{definition}
Given a set $S \subseteq \P_{\C}^{d-1}$, the collection of all polynomials in $\C[x_1, \ldots, x_d]$ that vanish on $S$ is a homogeneous ideal, known as the \textbf{vanishing ideal} of $S$, and denoted as $\mathbf{I}(S)$. The variety $\mathbf{V}(\mathbf{I}(S))$ is the the smallest complex projective variety that contains $S$, known as the \textbf{Zariski closure} of $S$. 
\end{definition}

We refer the reader to \cite{CLO07} for the basics on ideals and varieties. In this paper we will be interested in the vanishing ideal of the joint image $\varphi_\mathcal{A}(\P_\R^3)$.

\begin{definition}
\label{def:MA}
The \textbf{multiview ideal} of $\mathcal{A}$, denoted $M_\mathcal{A}$, is the vanishing ideal of 
$\varphi_\mathcal{A}(\P_\R^3)$ in $\mathbb{C}[p_1, \ldots, p_n]$ 
where $p_i = (x_i,y_i,z_i)$ are the coordinates on the $i$th copy of $\P^2_\C$. The Zariski closure of 
$\varphi_\mathcal{A}(\P_\R^3)$ in $(\P_\C^2)^n$ is the complex projective variety 
$\mathbf{V}(M_\mathcal{A})$, which we call the \textbf{multiview variety} of $\mathcal{A}$.
\end{definition}

The terminology {\em multiview ideal} and {\em multiview variety} comes from \cite{AST11}. 
Following Triggs~\cite{triggs95}, Trager \etal refer to  the multiview variety as the {\em joint image variety}.

Starting with the seminal work of Longuet-Higgins~\cite{longuet1981computer}, researchers have studied various systems of polynomials that vanish on $\varphi_\mathcal{A}(\P_\R^3)$. In the computer vision literature these equations are known as {\em multiview constraints}~\cite{M12, FLP01, HZ, YM12, HA97}. Obviously, the ideals generated by these systems of polynomials are contained in $M_\mathcal{A}$. However, there hasn't been much discussion of whether these polynomials generate $M_\mathcal{A}$ since the focus of all these papers has been on the multiview variety and not its vanishing ideal. 
The aim of this paper is to provide a complete description of the multiview ideal and study its relationship to the above sets of  polynomials.

It can be difficult to determine the vanishing ideal of a variety. However, there are various advantages to 
knowing it. To be able to do any computations with a variety or to study its structure using algebra, we need a description in terms of polynomials and the vanishing ideal is the optimal algebraic description. This manifests itself in a number of ways.

The set of all polynomial functions on $X$ 
is precisely $\C[x_1, \ldots, x_d]/ \mathbf{I}(X)$, known as the coordinate ring of $X$. In particular,  a polynomial $g$ vanishes on $X$ if and only if $g$ belongs to
$\mathbf{I}(X)$. Knowledge of a generating set $\{g_1, \hdots, g_k\}$ of $\mathbf{I}(X)$ also informs us about the local structure of $X$, since a point $x \in X$ is {\em smooth} if and only if the Jacobian matrix $ (\frac{\partial{g_i}}{\partial{x_j}})$ has rank equal to the codimension of $X$.  More generally, if $X \subset \P_{\C}^{d-1}$ is a projective variety then $\mathbf{I}(X)$ carries all the geometric information about $X$ allowing algebra (and algebraic algorithms) to infer geometric properties of $X$. For example, the dimension and degree of $X$ can be read off from the {\em Hilbert polynomial} of $\mathbf{I}(X)$ which also carries many more sophisticated invariants of $X$. See \cite{CLO07} for all the above.

In multiview geometry, many estimation problems can be phrased as polynomial optimization problems over varieties \cite{kahlhenrion2007, AST11}. In particular, the triangulation problem under Gaussian noise amounts to projecting a point onto the multiview variety\cite{AAT12}. 

In general, polynomial optimization on a variety $X \subseteq \R^n$ boils down to certifying the non-negativity of a polynomial $f$ on $X$ by expressing it as a sum-of-squares (sos) modulo an ideal $J$ vanishing on $X$  ~\cite{blekherman2012semidefinite}. This means finding a sos  polynomial $s = \sum p_i^2$ such that $f-s$ lies in $J$. This expressibility is maximized, and the algorithms terminate in the lowest possible degree, when $J = \mathbf{I}(X)$. We illustrate this on a very small example.

\begin{example}
The polynomial $x+1$ is non-negative on $X = \{0\} \subset \R$. The ideal $\langle x^2 \rangle$ cuts out $X$ but $\mathbf{I}(X) = \langle x \rangle$. Now $(x+1)-1 \in \langle x \rangle$ allowing $s=1$ as the sos certificate. On the other hand, if $x+1-s \in \langle x^2 \rangle$ then $s$ has to have degree at least $2$; for instance $(x+1)-(1+\frac{1}{2}x)^2 \in \langle x^2 \rangle$.
\end{example}

The above phenomenon can have a major impact on the number of rounds of convex relaxations needed to solve a polynomial optimization problem such as the well-known {\em Lasserre/sos hierarchies} \cite{lasserre2000,parrilo2003}, where each round looks for sos certificates of a fixed degree with degrees increasing monotonically with rounds. In each round the semidefinite program being solved is of size $O(n^d)$, where $n$ is the number of variables and $d$ is degree in that round. As a result, in many cases only the first round maybe computationally feasible and having access to $\mathbf{I}(X)$ can make the difference between the problem being tractable or not.

The rest of the paper is structured as follows. After a brief discussion of the notation used in this paper we begin in Section~\ref{sec:kfocal ideals} by introducing a family of ideals associated with every camera arrangement $\mathcal{A}$ which we call the $k$-focal ideals. We describe how these ideals behave under change of coordinates, and dispel the popular myth that, under a change of image coordinates, $k$-focal polynomials go to $k$-focal polynomials. In Section~\ref{sec:multiview ideal}, we prove our first main theorem (Theorem~\ref{thm:BA+TA=JA}), that the well-known bifocal (epipolar constraints) and trifocal polynomials generate $M_\mathcal{A}$ when the camera foci in $\mathcal{A}$ are distinct. Next, in Section~\ref{sec:determinantal ideals}, we consider three different types of determinantal polynomials proposed to cut out the multiview variety by Heyden-\AA str\"{o}m~\cite{HA97},  Faugeras \etal~\cite{FLP01} and Ma \etal~\cite{YM12}. We show that while the ideals they generate are all contained in $M_\mathcal{A}$, none of them actually coincide with $M_\mathcal{A}$. We establish their precise algebraic relationship with $M_\mathcal{A}$. In Section~\ref{sec:bifocal ideal}, we consider the relationship of the multiview ideal to bifocal polynomials  and prove the algebraic analog of the statement that the bifocal polynomials cut out the multiview variety when the camera foci are noncoplanar. In Section~\ref{sec:finite images}, we study how the various ideals relate to each other when we restrict our attention to finite images, \ie exclude points at infinity. We conclude in Section~\ref{sec:summary} with a summary.

Many results in this paper require explicit computation. We recommend the reader have a copy of Macaulay2~\cite{grayson2002macaulay} (or equivalent symbolic algebra software) handy. The Macaulay2 codes for our computations can be found at \url{ https://sites.math.washington.edu/~thomas/papers/Multiview_Ideal.zip}

\subsection{Notation}
In the rest of the paper, we will use $\P$ to denote $\P_\C$. The ideal generated by the polynomials $f_1,\ldots,f_s$ will be denoted as 
$\langle f_1, \ldots, f_s \rangle$.

We will use $A$ for cameras and $G$ for matrices in $\textup{GL}_n$. $\mathcal{A}$ and $\mathcal{G}$ will denote arrangements of corresponding matrices. Bold, lower-case roman letters will be used to indicate vectors, and lower-case greek letters will be used for functions. Given a partial symbolic matrix $M$, $minors(k,M)$ will denote the ideal generated by all $k\times k$ minors of the matrix $M$. The symbol $[n]$ denotes the set $\{1,\hdots, n\}$ and $\binom{[n]}{m}$ denotes the set of all size $m$ subsets of $[n]$. 

\section{The $k$-focal ideals of a camera arrangement}
\label{sec:kfocal ideals}

Let $p_i$ be the tuple of variables $(x_i,y_i,z_i)$ denoting the coordinates associated to 
the projective plane $\mathbb{P}_\R^2$ corresponding to the $i$th camera image. Write $p = (p_1, \ldots, p_n)$, and consider the partially symbolic matrix 
\begin{align} \label{eq:HZmatrix}
\mathcal{A}(p) := \begin{bmatrix} 
A_1 & {p}_1 & & &\\
A_2 &  & {p}_2 &&\\
\vdots &  &  & \ddots &  \\
A_n &  &  &  & {p}_n
\end{bmatrix}. 
\end{align}
Let $\mathcal{A}(\mathbf{p})$ denote the evaluation of $\mathcal{A}(p)$ at $p = \mathbf{p}$. 
If $\mathbf{p} := (\mathbf{p}_1, \ldots, \mathbf{p}_n) \in \varphi_\mathcal{A}(\mathbb{P}_\R^3)$ then there exists 
some $\mathbf{q} \in \P_\R^3$ and scalars $\lambda_i \in \R$ such that $A_i \mathbf{q} = \lambda_i \mathbf{p}_i$ for all $i=1,\ldots,n$. Therefore, $\mathcal{A}(\mathbf{p})$ has a non-trivial kernel since it contains the point $(\mathbf{q}, -\lambda_1, \ldots, -\lambda_n)$, and hence the maximal minors of $\mathcal{A}(p)$, which are polynomials in 
$p_1, \ldots, p_n$, vanish on $\mathbf{p}$. Since $\mathbf{p}$ was arbitrary, 
these maximal minors vanish on all of $\varphi_\mathcal{A}(\P_\R^3)$ and on the multiview variety. Therefore, 
\begin{align*}
minors(4 + n , \mathcal{A}(p)) \subseteq M_\mathcal{A}. 
\end{align*}

%We just saw that the maximal minors of the partially symbolic matrix $\mathcal{A}(p)$ vanish on the multiview variety and are hence contained in the multiview ideal $M_\mathcal{A}$. 

In this section, we describe further minors of $\mathcal{A}(p)$ and the ideals they generate, which will play an important role in the description of $M_\mathcal{A}$. 

\begin{definition}  
\label{def:k focals}
For a subset $\si = \{\si_1, \dots, \si_k\} \subseteq [n]$ where $k \ge 2$, consider the partially symbolic matrix \begin{align}
%\label{eq:Asigma}
\mathcal{A}_\si(p) = \left(\BM A_{\si_1} &p_{\si_1} &0&\dots &0\\A_{\si_2} &0&p_{\si_2}&\ddots &0 \\ \vdots & \vdots &\ddots&\ddots &\vdots \\A_{\si_k} &0&\dots&0&p_{\si_k}\EM \right) 
\end{align}
of size $3k \times (4 + k)$. 
A maximal $(4 + k) \times (4 + k)$ minor of $\mathcal{A}_\si(p)$ is called a {\bf $k$-focal polynomial} of $\mathcal{A}$. The {\bf $k$-focal ideal} of $\mathcal{A}$, $H_\mathcal{A}^k$, is the ideal sum
\[
H_\mathcal{A}^k = \sum_{\sigma \in {[n]\choose k}} minors(4+k, \mathcal{A}_\sigma(p)).
\]

%\[
%H_\mathcal{A}^k = \sum_{\sigma \in {[n]\choose k}} ( (4+k) \times (4+k) \textup{ minors of } \mathcal{A}_\sigma(p)).
%\]
\end{definition}

Trager \etal also study the $k$-focal polynomials and refer to them as $k$-{\em linearities}~\cite{THP15,THPsupp15}.
% and derive them as minors of $\mathcal{A}(p)$. 
Note that every $k$-focal polynomial is multilinear and of total degree $k$. 
Such a minor involves choosing $4 + k$ rows of $\mathcal{A}_\si(p)$, and by a pigeonhole 
argument, at most four cameras may contribute more than one row to the minor when $k > 4$. 
Indeed, if more than four cameras contributed at least two rows each, then at least $10$ rows are accounted for, which leaves at most $k-6$ rows to take from the remaining $k-5$ cameras. So at least one camera will be left out entirely which means that the submatrix of that $4+k$ minor has a zero column and the minor is zero. 

A useful fact for us will be that for two positive integers $l >  k \geq 2$, there is a simple way to ``bump up'' a $k$-focal polynomial to an $l$-focal polynomial by multiplying the $k$-focal polynomial with a monomial. 
\begin{lemma}
\label{lem:bumping}
Suppose $f$ is a $k$-focal polynomial from cameras $\si = \{\si_1, \dots, \si_k\} \subset [n]$ where $k \geq 2$. For any $l > k$ cameras $\tau = \{\si_1 ,\dots, \si_k, \tau_1, \dots , \tau_{l-k} \}$, there is a $l$-focal polynomial $g$ such that $(\prod_{i= 1}^{l-k} w_{\tau_i}  ) f = g $ for any choice of variables $w_{\tau_i} \in \{x_{\tau_i}, y_{\tau_i}, z_{\tau_i} \}$, one for each camera. 
\end{lemma}

\begin{proof}
Add the row and column associated to coordinate $w_{\tau_i}$ to ${\mathcal{A}_\si}(p)$ for $ \tau_1, \dots , \tau_{l-k}$ as follows
\[
\left(\BM A_{\si_1} &p_{\si_1} &\dots &0 & 0  & \dots &0  \\ \vdots & \vdots & \ddots &\vdots & \vdots  &\ddots & \vdots \\A_{\si_k} &0 & \dots&p_{\si_k} & 0 & \dots & 0  \\ (A_{\tau_1})_{w_{\tau_{1}}} & 0 & \dots & 0&w_{\tau_1} & \dots & 0  \\ \vdots & \vdots & \ddots & \vdots & \vdots &\ddots &0 \\(A_{\tau_{l-k}})_{w_{\tau_{l-k}}} &  0 & \dots & 0 & 0 & \dots & w_{\tau_{l-k}}  \EM \right).   
\]
Taking the determinant of this matrix yields the $l$-focal polynomial $g = (\prod_{i= 1}^{l-k} w_{\tau_i}  ) f$.
\end{proof}
Combining the above facts we get that any $l$-focal polynomial for $l > 4$ is of the form $(\prod_{i= 1}^{l-k} w_{\tau_i}  ) f$ where $f$ is a $k\le 4$ focal polynomial. This is a generalization of 
Proposition 2 in \cite{THP15} that showed that every $n$-focal polynomial is a monomial multiple of a $k$-focal polynomial for $k \leq 4$. As a result, we will primarily focus on the ideals $H_\mathcal{A}^2$, $H_\mathcal{A}^3$, and $H_\mathcal{A}^4$, called the {\em bifocal}, {\em trifocal},  and {\em quadrifocal ideals} of $\mathcal{A}$. 

A closer look at $H^2_\mathcal{A}$ reveals that it is the ideal generated by the $n \choose 2$ epipolar constraints, since $\mathcal{A}_{\{i,j\}}$ is a $6 \times 6$ matrix, whose determinant is the epipolar constraint between images $i$ and $j$. By Lemma~\ref{lem:bumping}, $H^3_\mathcal{A}$  contains the bumped up version of $H^2_\mathcal{A}$ and for every triplet of images $\{i,j,k\}$, the 27 trifocals implied by the three trifocal tensors relating them. And finally, $H^4_\mathcal{A}$ contains the bumped up versions of  $H^2_\mathcal{A}$ and $H^3_\mathcal{A}$ and the 81 quadrifocals implied by the quadrifocal tensor. The fact that we only need to study $H_\mathcal{A}^2$, $H_\mathcal{A}^3$, and $H_\mathcal{A}^4$ lines up with the well known fact in multiview geometry that when studying $n$-view constraints, one only needs to study the epipolar matrix, the trifocal tensor and the quadrifocal tensor. See Chapter 17 in the book by Hartley \& Zisserman~\cite{HZ} for explicit computations of the generators of $H^2_\mathcal{A}, H^3_\mathcal{A},$ and $H^4_\mathcal{A}$ and their history. 

In the remainder of this section, we will investigate how $k$-focal ideals transform under certain linear transformations on cameras. It is widely known that, from image data, the geometry of a camera arrangement can only be determined up to an arbitrary choice of $\P^3$ coordinates. This is reflected in the following lemma.

\begin{lemma}[Projective Ambiguity] \label{lem:world coordinate change}
Suppose $G \in \GL_4$. Then for any $k$, $H_\mathcal{A}^k = H^k_{\mathcal{A}G}$ where 
$\mathcal{A}G = (A_1G, A_2G, \ldots, A_kG)$.
\end{lemma}

\begin{proof}
This follows since $(\mathcal{A}G)_\si(p) = \mathcal{A}_\si(p) \diag(G, I_k)$ for any $k$-element subset $\sigma \subset [n]$ which implies that any $k$-focal of $\mathcal{A}G$ differs from the same $k$-focal of $\mathcal{A}$ by a factor of $\det(G) \neq 0$. 
\end{proof}

From the proof of Lemma~\ref{lem:world coordinate change}, we see that a $\P^3$ coordinate change 
that sends $\mathbf{q} \mapsto G \mathbf{q}$ maps $k$-focals to $k$-focals, picking up only a scalar factor $\det G \neq 0$. We will now see that change of coordinates on the image planes $\P^2$ affect the $k$-focals in a more subtle way. 

Let $\mathcal{G} =(G_1, \ldots, G_n) \in (GL_3)^n$ be a sequence of invertible matrices and consider the camera arrangement $\mathcal{G} \mathcal{A} := (G_1 A_1, \ldots, G_n A_n)$ obtained from a given 
arrangement $\mathcal{A}$ by left-multiplying $A_i$ with $G_i$. Note that the focal point of the camera 
$A_i$ is the same as the focal point of the camera $G_iA_i$. Since $p_i = (x_i, y_i, z_i)$, we 
denote the ring $\mathbb{C}[x_1, y_1, z_1, \ldots, x_i, y_i, z_i, \ldots, x_n,y_n,z_n]$ by 
$\mathbb{C}[p_1, \ldots, p_n]$ and a polynomial in it by $f(p_1, \ldots, p_n)$.
The sequence $\mathcal{G}$ induces a camera-wise linear change of coordinates $  \chi_\mathcal{G}$ on 
$\mathbb{C}[p_1, \ldots, p_n]$ by 
sending
\begin{align}
\label{eq:chiG}
\chi_\mathcal{G} : \begin{pmatrix} x_i \\ y_i \\ z_i \end{pmatrix} \mapsto G_i^{-1} \begin{pmatrix} x_i \\ y_i \\ z_i \end{pmatrix}
\end{align}

Note that this amounts to a change of coordinates in the image planes $\P^2$ of the cameras in $\mathcal{A}$. Let 
$G^{-1}p$ denote $\chi_\mathcal{G}(p) = (G_1^{-1}p_1, \ldots, G_n^{-1}p_n)$. In what follows we will also need the notation $\mathcal{G}^{-1} := (G_1^{-1}, \ldots, G_n^{-1})$,  $\mathcal{G}^{-1} \mathcal{A} := 
(G_1^{-1} A_1, \ldots, G_n^{-1} A_n)$ and $\chi_{\mathcal{G}^{-1}}(p_i) = G_i p_i$.

To analyze the effect of $\chi_\mathcal{G}$ on $k$-focal ideals, we recall the classical Cauchy-Binet formula, a proof of which can be found in~\cite{BW89}.
 
\begin{lemma}[Cauchy-Binet]
\label{lem:Cauchy Binet}
If $A$ and $B$ are rectangular matrices of size $m \times n$ and $n \times m$, respectively, where $m \leq n$, then the determinant of the square matrix $AB$ is:
$$ \det(AB) = \sum_{\sigma \in \binom{[n]}{m}} \det(A_{[:,\sigma]}) \det(B_{[\sigma, :]}) $$
where $:$ indicates that all rows/columns are taken. 
\end{lemma}

\begin{lemma} \label{lem:kfocal ideals map} For the $k$-focal ideal $H_\mathcal{A}^k$, 
$  \chi_\mathcal{G}(H_\mathcal{A}^k) = H_{\mathcal{G}\mathcal{A}}^k$. Similarly,  
$  \chi_{\mathcal{G}^{-1}}(H_{\mathcal{G}\mathcal{A}}^k) = H_\mathcal{A}^k$.
\end{lemma}

\begin{proof} 
We prove the first statement and the other follows similarly. 
We will show that the 
$k$-focal ideal of $\mathcal{A}_{[k]}$ is sent to the $k$-focal ideal of $(\mathcal{GA})_{[k]}$. The result then follows for the full $k$-focal ideal $H_\mathcal{A}^k$ by summing the $k$-focal ideals of 
all $\mathcal{A}_\sigma$ as $\sigma$ varies over all $k$-subsets of $[n]$.

Recall that a $k$-focal polynomial of $\mathcal{A}_{[k]} := (A_1, \dots, A_k)$ is a maximal minor of:
\begin{align*} %\label{eq:original}
\mathcal{A}_{[k]}(p) =  \begin{bmatrix} 
A_1 & p_1 & & &\\
A_2 &  & p_2 &&\\
\vdots &  &  & \ddots &  \\
A_k &  &  &  & p_k
\end{bmatrix}.
\end{align*}
Applying $  \chi_\mathcal{G}$ to this maximal minor is the same as taking the same maximal minor of 
\begin{align*} %\label{eq:after change of coordinates}
\mathcal{A}_{[k]}(  \chi_\mathcal{G}(p)) =  \begin{bmatrix} 
A_1 & G_1^{-1}p_1 & & &\\
A_2 &  & G_2^{-1}p_2 &&\\
\vdots &  &  & \ddots &  \\
A_k &  &  & & G_k^{-1}p_k
\end{bmatrix}.
    \end{align*}
The corresponding $k$-focal polynomial of $\mathcal{G} \mathcal{A}$ is the same maximal minor of 
    \begin{align}
    (\mathcal{G}\mathcal{A})_{[k]}(p) = \diag(G_1, \dots, G_k) \mathcal{A}_{[k]}(p)
    \end{align}
    The ideal $  \chi_\mathcal{G}(H_{\mathcal{A}_{[k]}}^k)$ is generated by the maximal minors of $\mathcal{A}_{[k]}(   \chi_\mathcal{G}(p)) $, namely
    \[
    \left\{ \det(\mathcal{A}_{[k]}(G^{-1}p)_{[\sigma,:]}) \,:\, \sigma \in \binom{[3k]}{4+k} \right\}, 
    \]
    while $H^k_{(\mathcal{G}\mathcal{A})_{[k]}}$ is generated by the maximal minors of $(\mathcal{G}\mathcal{A})_{[k]}(p)$. We need to show that these ideals coincide.

    Let $G$ denote the block diagonal matrix with blocks $G_1, \ldots, G_n$.  A $(4+k)$-minor of $(\mathcal{G}\mathcal{A})_{[k]}(p)$ is the determinant of a submatrix with $4+k$ rows indexed by some $\tau \in 
    {[3k] \choose 4+k}$. 
    Such a submatrix has the form $G_\tau \mathcal{A}_{[k]}(G^{-1}p)$ where $G_\tau$ is the submatrix of $G$ consisting of the rows of $G$ indexed by $\tau$. By the Cauchy-Binet formula, 
    \begin{align*} \label{eq:CauchyBinet on k-focal}
    \det(G_\tau& \mathcal{A}_{[k]}(G^{-1}p)) =\\ & \sum_{\sigma \in {[3k] \choose 4+k}}\det ((G_\tau)_{[:, \sigma]}) \det(\mathcal{A}_{[k]}(G^{-1}p)_{[\sigma,:]}).
    \end{align*}
 
    This implies that $\det(G_\tau \mathcal{A}_{[k]}(G^{-1}p))$ lies in the ideal $ \chi_\mathcal{G}(H_{\mathcal{A}_{[k]}}^k)$, and hence, 
    $H^k_{(\mathcal{G}\mathcal{A})_{[k]}} \subseteq  \chi_\mathcal{G}(H_{\mathcal{A}_{[k]}}^k)$.
    
    The reverse containment follows by applying the same argument to $\mathcal{A}_{[k]}(p) = G^{-1}G \mathcal{A}_{[k]}(p) $ and $G\mathcal{A}_{[k]}(p)$ where 
    $G^{-1}$ is the block diagonal matrix with blocks $G_1^{-1}, \ldots, G_k^{-1}$. 
    
    Summing over all $k$ camera subsets, the result follows:
    \begin{align*}
    \chi_\mathcal{G}(H^k_{\mathcal{A}}) &=  \chi_\mathcal{G}(\sum_{\si \in {[n] \choose k}} H^k_{\mathcal{A}_\si} ) =\sum_{\si \in {[n] \choose k}}   \chi_\mathcal{G}(H^k_{\mathcal{A}_\si}) \\ &= \sum_{\si \in {[n] \choose k}} H^k_{(\mathcal{GA})_\si} =  H_{\mathcal{G}\mathcal{A}}^k.
    \end{align*}
    \end{proof}
    
    This proof shows that, contrary to popular belief, it is not true that $k$-focal polynomials go to $k$-focal polynomials under the change of coordinates given by $ \chi_\mathcal{G}$, but the ideals do as in Lemma~\ref{lem:kfocal ideals map}.

\section{The Multiview Ideal}
\label{sec:multiview ideal}

Recall from Definition~\ref{def:MA} that the multiview ideal $M_\mathcal{A}$ of the camera arrangement 
$\mathcal{A}$ is the vanishing ideal of $\varphi_\mathcal{A}(\P_\R^3)$, meaning that it is the set of all 
polynomials in $\C[p_1, \ldots, p_n]$ that vanish on $\varphi_\mathcal{A}(\P_\R^3)$. 
Since $\varphi_\mathcal{A}(\P_\R^3)$ is a subset of $(\P^2_\R)^n$,  $M_\mathcal{A}$ is, in fact, generated by polynomials with real coefficients\footnote{Let $h(x) = f(x) + i g(x)$ be a complex polynomial, where $f(x)$ and $g(x)$ are real polynomials. Then if $h(x)$ vanish on a set of real points, then so must $f(x)$ and $g(x)$.}.

The complex projective variety $\mathbf{V}(M_\mathcal{A}) \subset 
(\P^2)^n$, which is the complex Zariski closure of  $\varphi_\mathcal{A}(\P_\R^3)$, is the multiview variety of $\mathcal{A}$. 
One might wonder if it is better to study the real Zariski closure of $\varphi_\mathcal{A}(\P_\R^3)$ and its vanishing ideal since complex points in the multiview variety do not have any physical meaning, and hence no relevance to multiview geometry. 
However, observe that if the real Zariski closure was strictly smaller than the set of real points in $\mathbf{V}(M_\mathcal{A})$, then there would be a polynomial not in 
$M_\mathcal{A}$ that vanishes on 
$\varphi_\mathcal{A}(\P_\R^3)$, which would contradict that 
$M_\mathcal{A}$ is the vanishing ideal of $\varphi_\mathcal{A}(\P_\R^3)$. Therefore, $M_\mathcal{A}$ is also the vanishing ideal of the real Zariski closure of $\varphi_\mathcal{A}(\P_\R^3)$, and hence a {\em real radical ideal} \cite[\S 12.5]{MarshallBook}. 

Further, since $\varphi_\mathcal{A}$ is a polynomial map and $\P_\R^3$ is irreducible,  
$\mathbf{V}(M_\mathcal{A})$ is an irreducible three-dimensional variety in $(\P^2)^n$. Hence  $M_\mathcal{A}$ is a  {\em prime} (homogeneous) ideal, meaning that if $fg \in M_\mathcal{A}$ then either $f$ or $g$ is in $M_\mathcal{A}$. 

It was shown in \cite{AST11} that the bifocals, trifocals and quadrifocals of 
$\mathcal{A}$ form a {\em universal Gr\"obner basis} of $M_\mathcal{A}$ under a certain genericity assumption on the cameras. 
This means that this collection of polynomials form a {\em Gr\"obner basis} for $M_\mathcal{A}$ with respect to any term order \cite{CLO07}. We will use this result to establish a generating set for $M_\mathcal{A}$ when the camera foci are distinct. 

We first note what happens to $M_\mathcal{A}$ under the change of coordinates $\chi_\mathcal{G}$ defined in the previous section. Recall that $\chi_\mathcal{G}$ sends a polynomial 
$f(p_1, \ldots, p_n) \in \mathbb{C}[p_1, \ldots, p_n]$ to $f(G_1^{-1}p_1, \ldots, G_n^{-1} p_n)$.
%%%%%%%%%%%

\begin{lemma} \label{lem:coordinate change}
The image of  the multiview ideal $M_\mathcal{A}$ under the map $  \chi_{\mathcal{G}}$ is  $M_{\mathcal{G} \mathcal{A}}$, the multiview ideal of $\mathcal{G} \mathcal{A}$. \ie, $  \chi_{\mathcal{G}}(M_\mathcal{A}) = M_{\mathcal{G} \mathcal{A}}$. 
Similarly, $  \chi_{\mathcal{G}^{-1}}(M_{\mathcal{G} \mathcal{A}}) = M_\mathcal{A}$.
\end{lemma}

\begin{proof} 
Again, we will prove that $  \chi_{\mathcal{G}}(M_\mathcal{A}) = M_{\mathcal{G} \mathcal{A}}$. The proof that 
$  \chi_{\mathcal{G}^{-1}}(M_{\mathcal{G} \mathcal{A}}) = M_\mathcal{A}$ is similar.

From the definition we see that a polynomial $f(p_1, \ldots, p_n)$ vanishes on the multiview variety $\mathbf{V}(M_\mathcal{A})$  if and only if $f(A_1 \mathbf{q}, \ldots, A_n \mathbf{q})=0$ for all $\mathbf{q} \in \mathbb{P}^3 \minus \{\mathbf{c}_1, \ldots, \mathbf{c}_n\}$, equivalently, if and only if 
$$f(G_1^{-1} (G_1 A_1 \mathbf{q}),  \ldots, G_n^{-1} (G_n A_n \mathbf{q}))=0$$
 for all $\mathbf{q} \in \mathbb{P}^3 \minus \{\mathbf{c}_1, \ldots, \mathbf{c}_n\}$.  
 The multiview variety of $\mathcal{G} \mathcal{A}$ is the Zariski closure of the points $( G_1A_1 \mathbf{q}, \ldots, G_nA_n \mathbf{q})$ as $\mathbf{q}$ varies over $\mathbb{P}^3 \minus \{\mathbf{c}_1, \ldots, \mathbf{c}_n\}$. Therefore, $f$ vanishes on 
 $\mathbf{V}(M_\mathcal{A})$ if and only if   
 $  \chi_\mathcal{G}(f)$ vanishes on $\mathbf{V}(M_{\mathcal{G} \mathcal{A}})$. This proves that 
 $  \chi_{\mathcal{G}}(M_\mathcal{A})  \subseteq  M_{\mathcal{G} \mathcal{A}}$.
 
 To finish the proof we need to argue that if $g(p_1, \ldots, p_n) \in M_{\mathcal{G}\mathcal{A}}$ then $g =   \chi_\mathcal{G}(f)$ for some $f \in M_\mathcal{A}$. A polynomial $g \in M_{\mathcal{G}\mathcal{A}}$ if and only if 
 $g(G_1A_1 \mathbf{q}, \ldots, G_n A_n \mathbf{q})=0$ for all $\mathbf{q} \in \mathbb{P}^3 \minus \{\mathbf{c}_1, \ldots, \mathbf{c}_n\}$ if and only if $g(G_1 \mathbf{p}_1, \ldots, G_n \mathbf{p}_n) = 0$ for all 
 $(\mathbf{p}_1, \ldots, \mathbf{p}_n) \in \mathbf{V}(M_\mathcal{A})$. 
 Define $g(G_1p_1, \ldots, G_n p_n) =: f \in M_\mathcal{A}$.
 Then $  \chi_\mathcal{G}(f) = g(p_1, \ldots, p_n)$.
\end{proof}
%%%%%%%%%%%%%

We will use the results obtained so far to give an elementary proof that the bifocals and trifocals generate the multiview ideal 
$M_\mathcal{A}$ for any arrangement $\mathcal{A}$ of cameras with pairwise distinct foci. 
An important tool will be {\em translational} cameras. 

\begin{definition}
A camera $T$ is said to be translational if its left $3\times 3$ block is the 
identity matrix, \ie, $T = [I \, \, \mathbf{t}]$ for some $\mathbf{t} \in \mathbb{R}^3$.
\end{definition}

\begin{lemma} \label{lem:Q inside T for translational cameras} 
If $\mathcal{T}$ is an arrangement of translational cameras, then 
$H^4_\mathcal{T} \subseteq H^3_\mathcal{T}$. 
\end{lemma}

\begin{proof}
Using Macaulay2, this statement can be checked for $n=4$ translational cameras with foci represented symbolically as $(t_{i1}, t_{i2},t_{i3}, -1)$.
%\begin{align*}
% \left[ \BM 1&0&0& t_{11} \\ 0&1&0&t_{12} \\0&0&1 & t_{13} \EM \right], \left[ \BM 1&0&0& t_{21} \\ 0&1&0&t_{22} \\0&0&1 & t_{23} \EM \right], \left[ \BM 1&0&0& t_{31} \\ 0&1&0&t_{32} \\0&0&1 & t_{33} \EM \right], \left[ \BM 1&0&0& t_{41} \\ 0&1&0&t_{42} \\0&0&1 & t_{43} \EM \right].
%\end{align*}
%Since this is a symbolic calculation, the result holds without any conditions on the $t_{ij}$'s.
For $n \geq 4$, since $H^4_{\mathcal{T}} = \sum_{\si \in {[n] \choose 4} } H^4_{\mathcal{T}_\si}$ and $H^3_\mathcal{T} = \sum_{\si \in {[n] \choose 3} } H^3_{\mathcal{T}_\si}$, the statement follows.
\end{proof}

We now use translational cameras to show that the quadrifocals are not needed in a generating set of $M_\mathcal{A}$. This is done by extending the result for translational cameras to finite cameras. Recall that a \textit{finite} camera is a camera whose left $3\times 3$ block is invertible, or equivalently a camera whose focal point is not a point at infinity. Observe that any finite camera can be obtained by multiplying some translational camera on the left by an invertible $3\times 3$ matrix.

\begin{corollary}
\label{cor:Q inside T for general cameras}
If $\mathcal{A}$ is any arrangement of cameras, then $H^4_\mathcal{A} \subseteq H^3_\mathcal{A}$. 
\end{corollary}
\begin{proof}
If  $\mathcal{A}$ is an arrangement of finite cameras, then 
$A_i = G_i [I \,\, \mathbf{t}_i]$ for some $G_i \in GL_3$. Therefore $\mathcal{A} = \mathcal{G} \mathcal{T}$ where 
$\mathcal{T}$ is an arrangement of translational cameras. By Lemma~\ref{lem:Q inside T for translational cameras}, $H^4_\mathcal{T} \subseteq H^3_\mathcal{T}$. Hence, Lemma~\ref{lem:kfocal ideals map} implies
\[
H^4_\mathcal{A} = H^4_{\mathcal{G}\mathcal{T}} = \chi_\mathcal{G}(H^4_\mathcal{T}) \subseteq \chi_\mathcal{G}(H^3_\mathcal{T}) = H^3_{\mathcal{G}\mathcal{T}} = H^3_\mathcal{A}.
\]

For any four cameras indexed by $\si \in { [n] \choose 4}$, there exists some $G \in \GL_4$ which takes the foci of $\mathcal{A}_\si$ off of the plane at infinity, \ie, so that $\mathcal{A}_\si G$ is an arrangement of finite cameras. Inverting this $\P^3$-coordinate change does not change ideal containment by Lemma~\ref{lem:world coordinate change}. The general result follows  since $H_\mathcal{A}^4 = \sum_{\si \in {[n] \choose 4}} H_{\mathcal{A}_\si}^4 \subseteq \sum_{\si \in {[n] \choose 3}} H_{\mathcal{A}_\si}^3 =H^3_\mathcal{A}.$
\end{proof}

To get to our main result, we will need a result from \cite{AST11} about camera arrangements $\mathcal{A}$ that are generic in the sense that all $4 \times 4$ minors of $[A_1^\top \, A_2^\top \, \cdots \, A_n^\top]$ are non-zero. We call such an $\mathcal{A}$ {\em minor-generic}.

\begin{corollary}
\label{cor:minorgeneric gives BA+TA}
Suppose $\mathcal{A}$ is minor-generic. Then $M_\mathcal{A} = H_\mathcal{A}^2 + H_\mathcal{A}^3$.
\end{corollary}
\begin{proof}
 Theorem 2.1 in \cite{AST11} says that if $\mathcal{A}$ is minor-generic, then the bifocals, trifocals and quadrifocals form a 
universal Gr\"obner basis of $M_\mathcal{A}$. In particular, this implies that $M_\mathcal{A} = H_\mathcal{A}^2 + H_\mathcal{A}^3  + H_\mathcal{A}^4$. The statement is then immediate from Corollary~\ref{cor:Q inside T for general cameras}.
\end{proof}

Minor-genericity is a purely algebraic condition on camera arrangements. The following statement, which appears as a brief comment in \cite{AST11} without proof, gives a geometric reinterpretation of this condition. 

\begin{lemma} \label{lem:minor-generic} 
If $\mathcal{A}$ is minor-generic, then the foci of the cameras in $\mathcal{A}$ are pairwise distinct. 
Conversely, if the cameras in $\mathcal{A}$ have pairwise distinct foci, then there exist $G_i \in \textup{GL}_3$ such that $\mathcal{G} \mathcal{A}$ is minor-generic.
\end{lemma}

\begin{proof} Let $L_i \subset \mathbb{C}^4$ denote the three-dimensional row span of $A_i$. 
If $A_i$ and $A_j$ have the same focal point then $L_i = L_j$ and hence any four of the six rows 
of $A_i$ and $A_j$ are linearly dependent and $\mathcal{A}$ is not minor-generic. This proves the first statement.

Now suppose the foci of cameras in $\mathcal{A}$ are pairwise distinct. This means that the planes $L_i$ are pairwise distinct. For any $G_i \in \textup{GL}_3$, the 
rows of $G_iA_i$ form a basis of $L_i$. By choosing $G_i$ appropriately, the three rows of $A_i$ can be sent to any choice of three linearly independent vectors in $L_i$. We need to show that there is a choice of $G_i$ such that no four rows from the matrices 
$G_iA_i$ are linearly dependent.

Consider the $3n \times 4$ matrix obtained by vertically stacking the cameras in $\mathcal{A}$, as a point in $(\mathbb{C}^4)^{3n}$, with coordinates $x_{kl}^i$ representing the $(k,l)$-entry of the $i$th camera.  We will identify this point in $(\mathbb{C}^4)^{3n}$ with the corresponding $3n \times 4$ matrix, and stack of $n$ cameras, and call all of them $\mathcal{A}$. Let $\mathcal{A}(x)$ denote the symbolic $3n \times 4$ matrix with entries $x^i_{kl}$. 
For $\sigma \in {[3n] \choose 4}$, let  $d_\sigma$ denote the determinant of the $4 \times 4$ submatrix of $\mathcal{A}(x)$ with rows indexed by $\sigma$.
These cut out ${3n \choose 4}$ quartic hypersurfaces $\mathbf{V}(d_\si)$ in $(\mathbb{C}^4)^{3n}$. Let $v_i$ denote the normal of the hyperplane $L_i \subset \mathbb{C}^4$. 
Impose linear conditions saying that the rows  of $\mathcal{A}(x)$, numbered $3i, 3i+1, 3i+2$, dot to zero with $v_i$. These $3n$ equations determine a subspace $L$ in $(\mathbb{C}^4)^{3n}$ of dimension at least $9n=12n-3n$. The given point $\mathcal{A}$ lies in $L$. We need to show that there is a choice of $\mathcal{G} \in (\textup{GL}_3)^n$ such that 
$\mathcal{G} \mathcal{A}$ (which again lies in $L$) avoids the determinantal surfaces. This is equivalent to picking a basis for each $L_i$ that stack together to 
a $\mathcal{B} \in L \minus \bigcup_\sigma \mathbf{V}(d_\sigma)$.

We first show that $L$ is not contained in any $\mathbf{V}(d_\sigma)$ by exhibiting a point in
$L \minus \mathbf{V}(d_\si)$ for each $\si$. Since at most four cameras can be involved in 
any $d_\si$, we may assume without loss of generality that $\sigma$ involves only rows of the 
first four cameras. There are four cases to consider depending on how many rows these four cameras contribute to $\sigma$ --- the possibilities being  
$(3,1,0,0)$, $(2,2,0,0)$, $(2,1,1,0)$, and $(1,1,1,1)$. In each case we will produce a 
$\mathcal{B} \in L \minus \mathbf{V}(d_\si)$. A key observation is that $A_i$ and $A_j$ having distinct foci implies $L_i \cap L_j$ is a proper subspace of both $L_i$ and $L_j$ for all $i,j$. 
Our starting point in each case below is $\mathcal{A} \in L$ which we modify to the needed $\mathcal{B}$ by replacing the bases of $L_i$ that provide the rows of $A_i$.

Case 1. \; \; (3,1,0,0): Modify $\mathcal{A}$ to $\mathcal{B}$ by choosing 
a basis for $L_2$ to be the three rows of $B_2$ so that 
no element in this basis lies in $L_1 \cap L_2$. Then  
$\mathcal{B}$ does not vanish on $d_\si$.

Case 2. \; \; (2,2,0,0): Choose a basis for $L_1$ such that the two rows 
$v_1, v_2$ contributing to $\si$ from the first camera are chosen from $L_1 \setminus L_2$. Then $L_2 \cap \Span\{v_1,v_2\}$ is a proper subspace of $L_2$ of dimension at most one. Therefore taking two linearly independent vectors $v_3,v_4$ outside of this subspace as the two rows from $L_2$ creates a $\mathcal{B}$ that does not vanish on $d_\si$.

Case 3. \;\; (2,1,1,0): Choose a basis for $L_1$ such that the two contributing rows $v_1, v_2$ from the first camera lie in $L_1 \setminus (L_2 \cup L_3)$. Choose the row $v_3$ from $L_2$ such that $v_3 \in L_2 \setminus (\Span\{v_1, v_2\} \cup L_3)$, which forces $L_3 \cap \Span\{v_1,v_2, v_3\}$ to be a proper subspace of $L_3$. Taking $v_4$ outside this subspace, we get a point $\mathcal{B} \in L$ at which $d_\si$ does not vanish.

Case 4. \;\; (1,1,1,1):  Choose $v_1 \in L_1 \minus (L_2 \cup L_3 \cup L_4)$, $v_2 \in L_2\minus (\Span\{v_1\} \cup L_3 \cup L_4)$, $v_3 \in L_3 \minus (\Span\{v_1,v_2 \} \cup L_4)$, and $v_4 \in L_4\minus (\Span\{v_1, v_2,v_3\} )$. By construction, we get a point in $L$ at which $d_\si$ does not vanish.

Therefore, $L \cap \mathbf{V}(d_\sigma)$ is a proper subvariety of $L$ for each $\sigma$, and 
a generic choice of $\mathcal{G}$ will put $\mathcal{G} \mathcal{A} \in L \minus \bigcup_\sigma \mathbf{V}(d_\sigma)$.
\end{proof}

We note that $\mathcal{A}$ having distinct foci does not imply that $\mathcal{A}$ is minor-generic.
A simple example would be an arrangement of four translational cameras; the submatrix consisting of the 
four first rows in each camera has zero determinant. However,  having distinct foci allows the camera arrangement to be made minor-generic by the action of a tuple $\mathcal{G}$. We are now ready to prove the main theorem of this section.

\begin{theorem} \label{thm:BA+TA=JA} 
Let $\mathcal{A}$ be an arrangement of cameras with 
distinct foci. Then $M_\mathcal{A} = H^2_{\mathcal{A}} + H^3_{\mathcal{A}}$.
 
\end{theorem}

\begin{proof}
By Lemma~\ref{lem:minor-generic}, there exists $\mathcal{G} \in (GL_3)^n$ such that $\mathcal{G} \mathcal{A}$ is minor-generic. Then, by Corollary~\ref{cor:minorgeneric gives BA+TA}, 
$M_{\mathcal{G}\mathcal{A}} = H^2_{\mathcal{G}\mathcal{A}} + H^3_{\mathcal{G}\mathcal{A}}$. 
%Theorem 2.1 in \cite{AST11} then implies that $M_{\mathcal{G}\mathcal{A}} = H^2_{\mathcal{G}\mathcal{A}} + H^3_{\mathcal{G}\mathcal{A}}+ H^4_{\mathcal{G}\mathcal{A}}$. 
 Therefore, by Lemmas~\ref{lem:coordinate change} and~\ref{lem:kfocal ideals map}, we get 
 \begin{align*}
M_{\mathcal{A}} = \chi_{\mathcal{G}^{-1}}(M_{\mathcal{G}\mathcal{A}}) &= 
 \chi_{\mathcal{G}^{-1}}(H^2_{\mathcal{G}\mathcal{A}})+
 \chi_{\mathcal{G}^{-1}}(H^3_{\mathcal{G}\mathcal{A}}) \\
 %+\chi_{\mathcal{G}^{-1}}(H^4_{\mathcal{G}\mathcal{A}})
 &=  H^2_{\mathcal{A}} + H^3_{\mathcal{A}}
 %+ H^4_{\mathcal{A}}
 \end{align*}
%Now we use Corollary~\ref{cor:Q inside T for general cameras} to conclude that $M_\mathcal{A} = H^2_\mathcal{A} + H^3_\mathcal{A}$.
\end{proof}

%The fact that these polynomials cut out $\mathbf{V}(M_\mathcal{A})$ can be found in 
%\cite[Proposition 5]{THP15}. This is in turn is based on a similar statement from \cite{HA97}. 

Proposition 5(1) in \cite{THP15} says that the $H^2_\mathcal{A}$ and $H^3_\mathcal{A}$ together cut out the multiview variety which implies that $H^2_\mathcal{A} + H^3_\mathcal{A} \subseteq \mathcal{M}_\mathcal{A}$. Theorem~\ref{thm:BA+TA=JA} shows that these polynomials also generate the multiview ideal providing the analogous ideal-theoretic statement. 

%It is mentioned at several points in \cite{THP15} that since \cite{AST11} proves that the bifocals, trifocals and quadrifocals form a universal Gr\"obner basis for $M_\mathcal{A}$, they generate $M_\mathcal{A}$. This is not immediate from \cite{AST11}. Certainly, quadrifocals can be ignored as they can be generated by the others, a fact that has been noted by several authors, and which we prove using translational cameras and change of coordinates (Corollary~\ref{cor:Q inside T for general cameras}). But, the universal Gr\"obner basis result in \cite{AST11} needs $\mathcal{A}$ to be minor-generic which requires more care

Theorem~\ref{thm:BA+TA=JA} improves on Corollary 2.7 in \cite{AST11} which states that 
when the foci of the cameras $A_i$ are in linearly general position, then $M_\mathcal{A}$ is generated by the bifocals and trifocals. Theorem~\ref{thm:BA+TA=JA} requires no sophisticated condition on the cameras beyond 
the foci being pairwise distinct. 

%Doing so requires showing that $\mathcal{A}$ with distinct foci can be made minor-generic by a change of coordinates in each image plane (Lemma~\ref{lem:minor-generic}), and that this transformation takes $k$-focal ideals to $k$-focal ideals (Lemma~\ref{lem:kfocal ideals map}).  

%Together, these results show that bifocals and trifocals generate $M_\mathcal{A}$, but the various steps needed are not immediately clear and require proofs. As we point out at the end of Section~\ref{sec:kfocal ideals},  it is not true, as is the popular belief, that each $k$-focal goes to a $k$-focal under the change of coordinates, but the ideals do. 

Conca \etal~\cite{conca2017cartwright} and Li~\cite{L18} also consider the vanishing ideal of the image of linear map from a projective space to a product of projective spaces. It is shown in~\cite{conca2017cartwright} that this ideal is Cartwright-Sturmfels, meaning that its initial ideal is radical after a generic change of coordinates. Both of these works allow for  projective spaces of arbitrary dimension. Specializing to our situation, Li's results show that $M_\mathcal{A} = \sum_{k = 2}^n H_\mathcal{A}^k$ while we prove that $M_\mathcal{A} = H_\mathcal{A}^2 + H_\mathcal{A}^3$.

Just like in \cite{THP15} where the results automatically generalized from 
projective cameras to Euclidean cameras, Theorem~\ref{thm:BA+TA=JA} also generalizes to Euclidean cameras. Recall that a camera $A_i$ is Euclidean if it is of the form $A_i=[R_i \,\, t_i]$ where $R_i \in \textup{SO}_3$.

\begin{corollary} \label{cor:Euclidean cameras} Let $\mathcal{A}$ be an arrangement of Euclidean cameras with pairwise distinct foci. Then $M_\mathcal{A} = H^2_\mathcal{A} + H^3_\mathcal{A}$.
\end{corollary}

We state one more consequence of Theorem~\ref{thm:BA+TA=JA} which will be needed in the next section. 

\begin{corollary}
\label{cor:limitpoints}
Let $\mathcal{A}$ be a camera arrangement with pairwise distinct foci. Then for any $\mathbf{p}_i \in \P^2$, the points 
$(A_1 \mathbf{c}_i , A_2 \mathbf{c}_i , \dots, \mathbf{p}_i, \dots, A_n \mathbf{c}_i)$ lie in $\mathbf{V}(M_\mathcal{A})$ 
where $\mathbf{c}_i$ is the focal point of $A_i$.
\end{corollary}

\begin{proof}
By Theorem~\ref{thm:BA+TA=JA}, it suffices to show that for any $i$, the bifocals and trifocals vanish on the points 
$(A_1 \mathbf{c}_i , A_2 \mathbf{c}_i , \dots, \mathbf{p}_i, \dots, A_n \mathbf{c}_i)$. For any pair of cameras $\{i, j\}$, observe that $(\mathbf{c}_i,  0, -1)$ is a nonzero element of $\ker\mathcal{A}_{\{i, j\}} (\mathbf{p}_i,  A_j \mathbf{c}_i)$. For any pair $\{j,k\}$ not containing camera $i$, $(\mathbf{c}_i,  -1, -1)$ is a nonzero element of $\ker\mathcal{A}_{\{j,k\}}(A_j \mathbf{c}_i,  A_k \mathbf{c}_i)$. Hence all polynomials of $H_\mathcal{A}^2$ vanish on $(A_1 \mathbf{c}_i , A_2 \mathbf{c}_i , \dots, \mathbf{p}_i, \dots, A_n \mathbf{c}_i)$. A similar argument applies to any triples of cameras, from which it follows that all polynomials in $H_\mathcal{A}^3$ vanish on $(A_1 \mathbf{c}_i , A_2 \mathbf{c}_i , \dots, \mathbf{p}_i, \dots, A_n \mathbf{c}_i)$. 
\end{proof}

The image of focal point $i$ in image $j$, \ie, $A_j \mathbf{c}_i$, is called the {\em epipole in image $j$ relative to image $i$}. Corollary~\ref{cor:limitpoints} shows that while the product of an arbitrary point in image $i$ with all epipoles relative to image $i$ does not appear in the image of $\phi_\mathcal{A}$, these points appear in the multiview variety after taking Zariski closure. See also Proposition 1 in~\cite{THP15}.

We conclude this section by showing that the hypothesis in Theorem~\ref{thm:BA+TA=JA} cannot be relaxed, namely if a pair of foci of cameras in $\mathcal{A}$ coincide, then the multiview ideal is strictly larger than the ideal generated by bifocals and trifocals.
\begin{example}
\label{ex:coincident foci2}
Consider the four translational camera arrangement $\mathcal{A}$ where $\mathbf{t}_1, \mathbf{t}_2 = (0,0,0)$, $ \mathbf{t}_3  = (1,1,1)$, $\mathbf{t}_4 = (-1,-1,-1)$. Eliminating the variables $q$ and $\la_i$ from the ideal $\<A_i q - \la_i p_i : i = 1, \dots, n\>$, we can directly obtain $M_\mathcal{A}$. Computing a primary decomposition of $H^2_\mathcal{A} + H^3_\mathcal{A}$, we find that 
\[H^2_\mathcal{A} + H^3_\mathcal{A} = M_\mathcal{A} \cap \langle {y}_{4}-{z}_{4},{y}_{3}-{z}_{3},{x}_{4}-{z}_{4},{x}_{3}-{z}_{3} \rangle.
\]
The extra component $\langle {y}_{4}-{z}_{4},{y}_{3}-{z}_{3},{x}_{4}-{z}_{4},{x}_{3}-{z}_{3} \rangle$ cuts out the points $(\mathbf{p}_1, \mathbf{p}_2, A_3\mathbf{c}_1, A_4 \mathbf{c}_1)$, and from the primary decomposition we see that the 
 projective variety they form is not contained in $\mathbf{V}(M_\mathcal{A})$.
\end{example}

\section{More Ideals for the Multiview Variety}
 \label{sec:determinantal ideals}
 
In the computer vision literature, there are several sets of polynomials that have been shown to vanish on the space of images $\varphi_\mathcal{A}(\P_\R^3)$, and hence they also vanish on the multiview variety. We now consider three such sets of polynomials and the ideals they generate, and compare them to the multiview ideal $M_\mathcal{A}$. 

\subsection{Heyden and \AA str\"om~\cite{HA97} }
\label{subsec:HA}
%The first ideal we will consider is the $n$-focal ideal $H^n_\mathcal{A}$. 
Heyden and \AA str\"om were the first to do an algebraic study of the multiview variety, by studying the $n$-focal ideal $H^n_\mathcal{A}$~\cite{HA97}. The variety of this ideal is indeed the multiview variety. 

\begin{lemma}
\label{lem:vanishingHA}
For any camera arrangement $\mathcal{A}$ with pairwise distinct foci, $\mathbf{V}(M_\mathcal{A}) = \mathbf{V}(H_\mathcal{A}^n)$. 
\end{lemma}

\begin{proof}
Recall from the image formation equations, $A_i \mathbf{q} = \lambda_i \mathbf{p}_i$ for all $i=1,\ldots,n$, that if 
$\mathbf{p} = (\mathbf{p}_1, \ldots, \mathbf{p}_n)$ lies in the image of $\varphi_\mathcal{A}$ then the matrix  $\mathcal{A}(\mathbf{p})$ has a non-trivial kernel. This means that all maximal minors of $\mathcal{A}(p)$ vanish on the image of $\varphi_\mathcal{A}$, and therefore also on its Zariski closure, which is the multiview variety. Therefore, 
$\mathbf{V}(M_\mathcal{A}) \subseteq  \mathbf{V}(H_\mathcal{A}^n)$. 

To see the reverse inclusion, suppose $\mathbf{p} = (\mathbf{p}_1, \ldots, \mathbf{p}_n) \in \mathbf{V}(H_\mathcal{A}^n)$ which means that $\mathcal{A}(\mathbf{p})$ is rank deficient and there is a nonzero 
vector of the form $(\mathbf{q},-\lambda_1, 
\ldots, -\lambda_n)$ in the kernel of $\mathcal{A}(\mathbf{p})$. If $\mathbf{q} = 0$, then 
we will get that $\lambda_i \mathbf{p}_i = 0$ for all $i$. However, since $\mathbf{p}_i \neq 0$, it must be that $\lambda_i=0$ for all $i$ and hence the vector in the kernel is the zero vector which is a contradiction. Therefore, there is a nonzero vector $\mathbf{q}$ such that $A_i \mathbf{q} = \lambda_i \mathbf{p}_i$ for some $\lambda_i$. If $\mathbf{q}$ is not the focal point of any camera, then $\mathbf{p}$ lies in $\varphi_\mathcal{A}(\P_\C^3)$. Since $\phi_\mathcal{A}$ is continuous, $\phi_{\mathcal{A}}(\bar{\P_\R^3})\subseteq \bar{\phi_\mathcal{A}(\P_\R^3)}$. It follows that $\phi_\mathcal{A}(\P^3_\C) \subseteq \mathbf{V}(M_\mathcal{A})$ because $\bar{\P_\R^3} = \P_\C^3$ and so $\mathbf{p} \in \mathbf{V}(M_\mathcal{A})$.  On the other hand, if $\mathbf{q}$ is the focal point $\mathbf{c}_i$ of camera $i$, then $\mathbf{p}_j = A_j\mathbf{c}_i$ for all $j \neq i$, and by  
%$\mathbf{p}$ is equal (projectively) to $(A_1 \mathbf{c}_i , A_2 \mathbf{c}_i , \dots, \mathbf{p}_i, \dots, A_n \mathbf{c}_i)$. 
Corollary~\ref{cor:limitpoints}, $\mathbf{p} \in \mathbf{V}(M_\mathcal{A})$. Thus we get that $\mathbf{V}(M_\mathcal{A}) \supseteq  \mathbf{V}(H_\mathcal{A}^n)$. 

\end{proof}

Example~\ref{ex:coincident foci2} shows that the assumption of distinct foci is necessary for Lemma~\ref{lem:vanishingHA}.
In this example, $n=4$ and $H^4_\mathcal{A} = H^2_\mathcal{A} + H^3_\mathcal{A}$ by Corollary~\ref{cor:Q inside T for general cameras}. We see that $\mathbf{V}(H^4_\mathcal{A})$ has a component other than $\mathbf{V}(M_\mathcal{A})$.

\subsection{Faugeras et al.~\cite{faugeras1995geometry}.}
The second set of polynomials we will study were constructed by Faugeras \& Mourrain while proving that the multiview variety is cut out by epipolar/bifocal and trifocal polynomials, and that the quadrifocal constraints corresponding to the quadrifocal tensor were not needed~\cite{FLP01,faugeras1995geometry}.

Observe that $A_i \mathbf{q} = \la_i \mathbf{p}_i$ implies $A_i \mathbf{q} \times \mathbf{p}_i = 0$, for each $i$, or equivalently, $[p_i]_\times A_i	\mathbf{q} = 0$, where 
\begin{align}
\label{eq:ptimes}
[p_i]_\times  = \left( \begin{matrix}  0 & -z_i & y_i \\ z_i & 0 & -x_i \\-y_i & x_i&0  \end{matrix} \right)
\end{align}
represents taking cross product with $p_i$, \ie, $[p_i]_\times v = p_i \times v$. Stacking all $3 \times 4$ matrices $[p_i]_\times A_i$, we get the $3n \times 4$ partially symbolic matrix
\begin{align}
\mathcal{A}^{F}(p) := \left( \begin{matrix}  [p_1]_\times A_1 \\ [p_2]_\times A_2 \\ \vdots \\ [p_n]_\times A_n \end{matrix} \right).
\end{align}
If there is a world point $\mathbf{q}$ satisfying $A_i \mathbf{q} \times \mathbf{p}_i = 0$, then this matrix is rank deficient and all maximal minors of $\mathcal{A}_F(p)$ vanishes on the multiview variety. 
\begin{definition}
The ideal of all maximal $4\times 4$ minors of $\mathcal{A}^{F}(p)$, denoted by $F_\mathcal{A}$, will be called the {\bf Faugeras ideal} of the arrangement $\mathcal{A}$. We denote the subideals of $F_\mathcal{A}$ generated by minors involving only two and three cameras by $F^2_\mathcal{A}$ and $F_\mathcal{A}^3$, respectively. 
%\marginpar{\textcolor{red}{do we use $F^2_\mathcal{A}$ and $F_\mathcal{A}^3$?} \textcolor{blue}{those are referenced but only in the Faugeras minors subsection} }
\end{definition}

We now describe a sequence of matrix transformations that allow us to obtain $\mathcal{A}^F(p)$ from $\mathcal{A}(p)$. Let $P(p) := \diag([p_1]_\times ,\dots, [p_n]_\times)$ be the symbolic block diagonal matrix of size $3n \times 3n$. Multiplying $\mathcal{A}(p)$ on the left by the block diagonal matrix $P(p)$ and dropping the rightmost $n$ columns of the resulting matrix, we obtain $\mathcal{A}^F(p)$:

\begin{align}
\label{eq:AFdecomposition}
 \mathcal{A}^F(p) = P(p) \mathcal{A}(p) \begin{bmatrix} I_4 \\ 0_{n \times 4} \end{bmatrix}  = P(p)  \mathcal{A}
\end{align}
where as before, we abuse notation to let $\mathcal{A}$ also represents the $3n \times 4$ matrix $[A_1; \dots ; A_n]$ obtained by stacking the cameras vertically. From the matrix constructions of $H_\mathcal{A}^n$ and $F_\mathcal{A}$, we observe that their projective vanishing sets in $(\P^2)^n$ coincide.

%\fixme{double check that we don't need to be concerned with epipoles here}

\begin{lemma}
\label{lem:vanishingFA}
For any camera arrangement $\mathcal{A}$ with pairwise distinct foci, $\mathbf{V}(M_\mathcal{A}) = \mathbf{V}(F_\mathcal{A}) $. 
\end{lemma}

\begin{proof}
The proof will follow from Lemma~\ref{lem:vanishingHA} if we can show that $\mathbf{V}(F_\mathcal{A}) = \mathbf{V}(H_\mathcal{A}^n)$. 
If $\mathbf{p} \in (\P^2)^n$ is such that $\mathcal{A}^F(\mathbf{p})$ drops rank, then there exists a nonzero $\mathbf{q} \in \ker( \mathcal{A}^F(\mathbf{p}) )$ so that $A_i \mathbf{q} \times \mathbf{p}_i = 0$ for all $i$. This means there exist nonzero scale factors $\lambda_i$ such that $A_i \mathbf{q} = \lambda_i \mathbf{p}_i$. The vector $(\mathbf{q}, - \la_1, \dots, -\la_n)$ is a nontrivial element in $\ker ( \mathcal{A}(\mathbf{p}) )$, so $\mathcal{A}(\mathbf{p})$ is rank deficient. Therefore $\mathbf{V}(F_\mathcal{A}) \subseteq \mathbf{V}(H^n_\mathcal{A})$

For the other inclusion, if there is a nontrivial $(\mathbf{q}, - \la_1, \dots, -\la_n) \in \ker( \mathcal{A}(\mathbf{p}) )$ for some $\mathbf{p} \in (\P^2)^n$, then as in the proof of Lemma~\ref{lem:vanishingHA}, $\mathbf{q}$ must be nonzero, and 
so $\mathbf{q}$ is a nontrivial element of $\ker( \mathcal{A}^F(\mathbf{p}) )$. This shows that $\mathbf{V}(F_\mathcal{A}) \supseteq \mathbf{V}(H^n_\mathcal{A})$, hence $\mathbf{V}(F_\mathcal{A}) = \mathbf{V}(H_\mathcal{A}^n)  = \mathbf{V}(M_\mathcal{A})  $. 

\end{proof}

\subsection{Ma et al.~\cite{YM12}}
The third and final set of polynomials we will study are the so called {\em multiview rank constraints} which were proposed by Ma and collaborators~\cite{YM12} as an alternative to the multilinear constraints studied 
for example in Hartley \& Zisserman~\cite{HZ}.

Suppose $A_1 = [I \, \, 0]$ and $A_i = [ B_i \, \, \mathbf{t}_i]$ for $i \ge 2$. Starting with $\mathcal{A}(p)$, a series of matrix operations are described in Chapter 8 in~\cite{YM12} to arrive at a new set of determinantal polynomials, arising as maximal minors of

\begin{align}
\mathcal{A}^{Y}(p) :=
 \begin{bmatrix} 
p_1 \times (I p_1) & p_1 \times 0 \\
p_2 \times (B_2 p_1) & p_2 \times t_2 \\
\vdots  & \vdots &  \\
\vdots& \vdots & \\
p_n \times (B_n p_1) & p_n \times t_n
\end{bmatrix}. 
\end{align}
\begin{definition}
The ideal of all maximal $2\times 2$ minors of $\mathcal{A}^{Y}(p)$, denoted by $Y_\mathcal{A}$, will be called the {\bf Ma ideal} of the arrangement $\mathcal{A}$. 
\end{definition}

We observe that $\mathcal{A}^Y(p)$ can be obtained from $\mathcal{A}^F(p)$ by multiplying by a single matrix on the right:
\begin{align}
\label{eq:AYdecomposition}
\mathcal{A}^Y(p) = \mathcal{A}^F(p) \begin{bmatrix} p_1 & 0 \\ 0&1 \end{bmatrix}. 
\end{align}
From this we observe that $Y_\mathcal{A}$ has the same projective vanishing set as 
$F_\mathcal{A}$, and hence $H_\mathcal{A}^n$ and $M_\mathcal{A}$. 

\begin{lemma}
\label{lem:vanishingYA}
For any camera arrangement $\mathcal{A}$ with pairwise distinct foci and $A_1 = [I\,\,0]$, $\mathbf{V}(M_\mathcal{A}) = \mathbf{V}(Y_\mathcal{A}) $. 
\end{lemma}

\begin{proof}
If $\mathbf{p} \in (\P^2)^n$ is such that $\mathcal{A}^Y(\mathbf{p})$ drops rank, then there exists a nontrivial $({v}_1, {v}_2) \in \ker ( \mathcal{A}^Y(\mathbf{p}) )$. Therefore, $\mathbf{q} = ({v}_1 \mathbf{p}_1, {v}_2) \in \ker ( \mathcal{A}^F(\mathbf{p}) )$ is nontrivial. Note that it is necessary that we assume $A_1 = [I\, \, 0]$ so that $[\mathbf{p}_1]_\times A_1 (v_1 \mathbf{p}_1, v_2) = v_1 [\mathbf{p}_1]_\times \mathbf{p}_1 = 0$. This shows that $ \mathbf{V}(Y_\mathcal{A}) \subseteq  \mathbf{V}(F_\mathcal{A})$.

For the other inclusion, if $0 \neq \mathbf{q} \in \ker (\mathcal{A}^F(\mathbf{p}) )$ for some $\mathbf{p} \in (\P^2)^n$, then since $\mathbf{p}_1 \times [I\, \, 0] \mathbf{q} = 0$, there exists a scalar $v_1$ such that $v_1 \mathbf{p}_1 = (\mathbf{q}_1, \mathbf{q}_2, \mathbf{q}_3)$. This means that $(v_1, \mathbf{q}_4) \in \ker ( \mathcal{A}^Y(\mathbf{p}) )$, which is nontrivial because if $v_1 = 0$, then $(\mathbf{q}_1, \mathbf{q}_2, \mathbf{q}_3) = 0$, so $\mathbf{q}_4 \neq 0$. This shows $\mathbf{V}(Y_\mathcal{A}) \supseteq \mathbf{V}(F_\mathcal{A})$, and the desired result follows from Lemma~\ref{lem:vanishingFA}. 

\end{proof}

Observe that $Y_\mathcal{A}$ is generated by polynomials of total degree 3. This fact has an interesting consequence. As we mentioned earlier, $Y_\mathcal{A}$ has been proposed as an alternate algebraic foundation for multi-view geometry. From Lemma~\ref{lem:vanishingYA}, we know that it cuts out the multiview variety. Since $M_\mathcal{A}$ is the vanishing ideal of the multiview variety, we get that $Y_\mathcal{A} \subseteq M_\mathcal{A}$. However, from Theorem~\ref{thm:BA+TA=JA} we know that $M_\mathcal{A} = H^2_\mathcal{A} + H^3_\mathcal{A}$, \ie it is generated by polynomials of degree two and three, which means 
that in general $Y_\mathcal{A} \neq M_\mathcal{A}$ and instead $Y_\mathcal{A} \subset M_\mathcal{A}$ or equivalently $Y_\mathcal{A} \subset H^2_\mathcal{A} + H^3_\mathcal{A} $.  This means that the bifocals and trifocals imply the multiview rank constraints, but not the other way around. Similarly, $H^n_\mathcal{A}$ and  $F_\mathcal{A}$,  which are generated by polynomials of total degree $n$ and four respectively, are 
properly contained in $M_\mathcal{A}$. We see this in Example~\ref{ex:normalized3} below.

%Example~\ref{ex:normalized3} shows that, in general, while the ideals $H_\mathcal{A}^n$, $F_\mathcal{A}$ and $Y_\mathcal{A}$ are all contained in $M_\mathcal{A}$, they are not the same. In the next subsection, we explore their differences.

%$H^n_\mathcal{A}$, $F_\mathcal{A}$, and $Y_\mathcal{A}$ are generated by polynomials of total degree $n$, four, and three respectively. The last fact has an interesting consequence.

\subsection{Relationships to the Multiview Ideal}\
\label{subsec:relationship to MA}
We now compute the three ideals on an example, foreshadowing their 
structural properties, which we examine next. 

\begin{example}
\label{ex:normalized3}
Consider the translational arrangement $\mathcal{A}$ where $\mathbf{t}_1 = (0,0,0)$, $ \mathbf{t}_2  = (1,0,0)$, $\mathbf{t}_3 = (0,1,0)$ whose multiview ideal is: 
\begin{align*}
&M_\mathcal{A} = \langle{y}_{1} {z}_{2}-{y}_{2} {z}_{1},{x}_{2} {z}_{3}-{x}_{3} {z}_{2}+{y}_{2} {z}_{3}-{y}_{3} {z}_{2},\\& {x}_{1} {z}_{3}-{x}_{3} {z}_{1},{x}_{1}
   {x}_{3} {y}_{2}+{x}_{1} {y}_{2} {y}_{3}-{x}_{2} {x}_{3} {y}_{1}-{x}_{3} {y}_{1} {y}_{2}\rangle.
 \end{align*}

The primary decompositions of $H^n_\mathcal{A}$, $F_\mathcal{A}$, and $Y_\mathcal{A}$ are
\begin{align*}
H_\mathcal{A}^n = M_\mathcal{A} &\cap \langle{z}_{1},{y}_{1},{x}_{1}\rangle \cap \langle {z}_{2},{y}_{2},{x}_{2}\rangle \cap  \langle{z}_{3},{y}_{3},{x}_{3}\rangle,  \\
 Y_\mathcal{A} = M_\mathcal{A} &\cap \langle{z}_{1},{y}_{1},{x}_{1}\rangle\cap \langle {y}_{3},{y}_{2},{x}_{3},{x}_{2},{z}_{3}^{2},{z}_{2} {z}_{3},{z}_{2}^{2}\rangle \\ &\cap \langle{z}_{1},{y}_{2},{x}_{3},{x}_{2},{x}_{1},{z}_{3}^{2},{z}_{2} {z}_{3},{z}_{2}^{2}\rangle \\& \cap \langle{z}_{1},{y}_{3},{y}_{2},{y}_{1},{x}_{3},{z}_{3}^{2},{z}_{2} {z}_{3},{z}_{2}^{2}\rangle,
\end{align*}%\vspace{-.15in}
\begin{align*}
F_\mathcal{A} = M_\mathcal{A} & \cap \langle{y}_{2},{y}_{1},{x}_{2},{x}_{1},{z}_{2}^{2},{z}_{1} {z}_{2},{z}_{1}^{2}\rangle \\& \cap \langle{y}_{3},{y}_{1},{x}_{3},{x}_{1},{z}_{3}^{2},{z}_{1}
       {z}_{3},{z}_{1}^{2}\rangle\\& 
\cap \langle{z}_{2},{z}_{1},{y}_{3},{y}_{2},{y}_{1},{x}_{3},{z}_{3}^{2}\rangle\\ & \cap \langle{z}_{3},{z}_{2},{y}_{1},{x}_{3}+{y}_{3},{x}_{2}+{y}_{2},{x}_{1},{z}_{1}^{2}\rangle \\
&\cap \left \langle {y}_{2},{x}_{2},{z}_{3}^{2},{z}_{2} {z}_{3},{z}_{2}^{2},{y}_{3} {z}_{3},{y}_{3} {z}_{2},{y}_{3}^{2},{x}_{3} {z}_{3},\right. \\
& \qquad \left . {x}_{3} {z}_{2},{x}_{3} {y}_{3},{x}_{3}^{2},{x}_{1}{x}_{3}+{x}_{1} {y}_{3}-{x}_{3} {y}_{1} \right \rangle \\ 
& 
\cap  \left \langle{z}_{3},{y}_{2},{x}_{2},{z}_{2}^{2},{z}_{1} {z}_{2},{z}_{1}^{2},{x}_{3} {z}_{2},{x}_{3}{z}_{1},{x}_{3}^{2},\right.\\
&\qquad \left. {x}_{1} {z}_{2},{x}_{1} {z}_{1},{x}_{1} {y}_{3}-{x}_{3} {y}_{1},{x}_{1} {x}_{3},{x}_{1}^{2}\right \rangle \cap C
\end{align*}
\normalsize{}
where $C$ is a component minimally generated by 133 polynomials of total degree up to eight.

While each of $H^n_\mathcal{A}$, $F_\mathcal{A}$, and $Y_\mathcal{A}$ notably contains $M_\mathcal{A}$ as a component, the nature of their other components is worth further investigation. 
\qed
\end{example}

To analyze the extra components, we rely on several notions from commutative algebra, which we define next. The first notion is that of a \textit{multigraded} ring. Consider the ring $\mathbb{C}[p_1, \dots, p_n]$ endowed with the $\mathbb{Z}^n$-grading $\deg (w_i) = \mathbf{e}_i$ where $w_i \in \{x_i,y_i,z_i\}$ and $\mathbf{e}_i$ is the $i$th standard basis vector in $\mathbb{R}^n$. We say a polynomial in this ring is homogeneous if each of its terms have the same multidegree.

The {\em irrelevant ideal} in this grading, which we denote by $\mathfrak{m}$, is the intersection of the ideals $\mathfrak{m}_i := \langle x_i,y_i,z_i\rangle$:
\begin{align}
%\label{eq:irrelevant}
\mathfrak{m} := \bigcap_{i= 1}^n \mathfrak{m}_i = \bigcap_{i=1}^n \langle x_i,y_i,z_i\rangle.
\end{align}
Observe that $\mathfrak{m}$ is generated by all multilinear monomials of multidegree $(1,1,\ldots,1)$ and total degree $n$. It is the maximal ideal in the ring $\mathbb{C}[p_1, \dots, p_n]$ generated by homogeneous elements of strictly positive multidegree.  

The \textit{radical} of an ideal $I$ is the ideal $\sqrt{I} := \{f : f^k \in I \text{ for some } k\in \mathbb{N} \}$. 
If $I$ is a homogeneous ideal then so is its radical, and $I \subseteq \sqrt{I}$.  
The \textit{colon} of an ideal $I$ with the ideal $J$, denoted as $(I:J)$ is the set of all polynomials $f$ such that $fg \in I$ for all $g \in J$, \ie, $I:J = \{ f \,:\, fJ \subseteq I \}.$   
%When $I$ is radical, taking the colon ideal of $I$ with $J$ grows the ideal $I$ so as to remove the components of $\mathbf{V}(I)$ that live in $\mathbf{V}(J)$.  \fixme{check this }

Recall that the projective varieties of the ideals $H^n_\mathcal{A}$, $F_\mathcal{A}$, and $Y_\mathcal{A}$ all agree and equal the multiview variety $\mathbf{V}(M_\mathcal{A})$. We can now state a first relationship among the ideals that follows easily from the projective Nullstellensatz in our multigraded setting, whose statement and proof will appear in Appendix A.

\begin{theorem}
\label{thm:radicalsaturated}
For any $\mathcal{A}$ with pairwise distinct foci,
\begin{enumerate}[label=\alph*)]
\item $\sqrt{H^n_\mathcal{A}} : \mathfrak{m} = M_\mathcal{A}$.
\item $\sqrt{F_\mathcal{A}} : \mathfrak{m} = M_\mathcal{A}$.
\item $\sqrt{Y_\mathcal{A}} : \mathfrak{m} = M_\mathcal{A}$ when $A_1 = [I \, \, 0]$.
\end{enumerate}
\end{theorem}

%\textcolor{red}{comment on $\P^3$ coordinate change applied to  $Y_\mathcal{A}$?}

%\textcolor{red}{now that this is in this order, is it clear that $(M_\mathcal{A} \cap \mathfrak{m} ): \mathfrak{m} = M_\mathcal{A} : \mathfrak{m} = M_\mathcal{A} : \mathfrak{m}$?}

\begin{proof}
See Appendix A.
\end{proof}

In the language of algebraic geometry what this says is that $\sqrt{H^n_\mathcal{A}}, \sqrt{F_\mathcal{A}}$ and 
$\sqrt{Y_\mathcal{A}}$ all cut out the multiview variety {\em scheme-theoretically}. They are not equal as ideals 
but they agree in high enough multidegree with $M_\mathcal{A}$, see \cite[pp 50]{H95}. 

%In the rest of this section, we assume that the foci of $\mathcal{A}$ are pairwise distinct.

We now strengthen Theorem~\ref{thm:radicalsaturated} (a) and (b) to show that the operation of taking the radical is not needed, \ie, $H^n_\mathcal{A} \,:\, \mathfrak{m} = M_\mathcal{A}$ and $F_\mathcal{A} \,:\, \mathfrak{m} = M_\mathcal{A}$. 
This means that $H^n_\mathcal{A}$ and $F_\mathcal{A}$ already cut out the multiview variety scheme-theoretically. 
Experimental evidence suggests that when $A_1 = [I\:|\; 0]$, such a result is also true for 
%a similar result to Theorems~\ref{thm:HZsaturated} and~\ref{thm:FAsaturated} should hold for 
$Y_\mathcal{A}$, but an explicit proof is made difficult by the convoluted structure of the $2\times 2$ minors of $\mathcal{A}^Y(p)$.

We first show that the simple structure of the primary decomposition of $H^n_\mathcal{A}$  observed in 
Example~\ref{ex:normalized3} holds in general.

\begin{lemma}
\label{lem:HZdecomposition}
 For any camera arrangement $\mathcal{A}$ with pairwise distinct foci, $H^n_\mathcal{A} = M_\mathcal{A} \cap \mathfrak{m}$. In particular, $H^n_\mathcal{A}$ is a radical ideal with prime decomposition $M_\mathcal{A} \cap \mathfrak{m}_1 \cap \mathfrak{m}_2 \cap \cdots \cap \mathfrak{m}_n$.
\end{lemma}

\begin{proof}
Suppose $f$ is a generator of $\in H^n_\mathcal{A}$, \ie, a maximal minor of $\mathcal{A}(p)$.  Then $f \in \mathfrak{m}$. Also, since $f$ vanishes on $\mathbf{V}(M_\mathcal{A})$, $f \in M_\mathcal{A}$. Therefore, 
 $H^n_\mathcal{A} \subseteq M_\mathcal{A} \cap \mathfrak{m}$.
 
 Now suppose $f \in M_\mathcal{A} \cap \mathfrak{m}$. Since $M_\mathcal{A}$ is generated by bifocals and trifocals 
$f = \sum \lambda_i r_i b_i + \sum \mu_j s_j t_j$ where $b_i$'s are bifocals, $t_j$'s are trifocals, $r_i, s_j$ are monomials,  and $\lambda_i, \mu_j$ are scalars. Further, since $f \in \mathfrak{m}$, every term in $f$ is divisible by some generator 
$\prod_{i=1}^n w_i$ of $\mathfrak{m}$ where $w_i \in \{x_i, y_i, z_i\}$. Now consider $r_ib_i$. Since $b_i$ involves only two cameras, it must be that $r_i$ contains a variable $w_i$ from each of the other $n-2$ cameras so that each term of $r_ib_i$ lies in $\mathfrak{m}$. This makes 
$r_ib_i$ a monomial multiple of a $n$-focal by Lemma~\ref{lem:bumping}. The same argument holds for $s_jt_j$. 
Thus, $f \in H^n_\mathcal{A}$. 
\end{proof}

Proposition b3 in \cite{THPsupp15} proves that when $\mathcal{A}$ is minor-generic, $H^n_\mathcal{A}$ is a radical ideal. Lemma~\ref{lem:HZdecomposition} shows that $H^n_\mathcal{A}$ is always a radical ideal under the weaker assumption of distinct foci.

\begin{theorem}
\label{thm:HZsaturated}
For any camera arrangement $\mathcal{A}$ with pairwise distinct foci,  $H^n_\mathcal{A} : \mathfrak{m} = M_\mathcal{A}$.
\end{theorem}

\begin{proof}
We first note that $M_\mathcal{A} :  \mathfrak{m} = M_\mathcal{A}$. Suppose $f \in M_\mathcal{A} : \mathfrak{m}$. 
Then $f u \in M_\mathcal{A}$ for any monomial generator $u$ of $\mathfrak{m}$. Since $M_\mathcal{A}$ is prime and does not contain any monomials, $f \in M_\mathcal{A}$. Since $H^n_\mathcal{A} = M_\mathcal{A} \cap \mathfrak{m}$ by Lemma~\ref{lem:HZdecomposition}, $H^n_\mathcal{A} : \mathfrak{m} = M_\mathcal{A} : \mathfrak{m} = M_\mathcal{A}$.
\end{proof}

We now consider the Faugeras ideal $F_\mathcal{A}$ and prove that $F_\mathcal{A} : \mathfrak{m} = M_\mathcal{A}$. The nontrivial part is to argue that $M_\mathcal{A}$ is contained in $F_\mathcal{A} : \mathfrak{m}$. This fact relies on the following technical lemma, similar in flavor to Lemma~\ref{lem:bumping}, which shows that bifocals and trifocals can both be multiplied by any generator of $\mathfrak{m}$ to fall into $F_\mathcal{A}$.  

\begin{lemma}
\label{lem:faugerasbumping}
\begin{enumerate}[label=\alph*)]
\item For $n=2$ cameras, and any monomial $p_{1j}p_{2k}$, there exists a $4 \times 4$ minor $f$ of $\mathcal{A}^F(p)$ such that $f = (-1)^{j+k} p_{1j}p_{2k} \det \mathcal{A}(p)$.
\item Let $n=3$ and $i_1,i_2,i_3$ be pairwise distinct. Then for any trifocal $ \det \mathcal{A}(p)_{\{p_{i_1j_1}p_{i_2j_2} \}}$ and any coordinate $p_{i_3k}$, there exists 
a $4 \times 4$ minor $f$ of $\mathcal{A}^F(p)$ such that 
$f = (-1)^k p_{i_3k} \det \mathcal{A}(p)_{\{p_{i_1j_1}p_{i_2j_2} \}}$.
\end{enumerate}
\end{lemma}

\begin{proof}
See Appendix B both for the notation and the proof.
\end{proof}

\begin{theorem}
\label{thm:FAsaturated}
For any camera arrangement $\mathcal{A}$ with pairwise distinct foci, $F_\mathcal{A} : \mathfrak{m} = M_\mathcal{A}$. 
\end{theorem}

\begin{proof}
The containment $F_\mathcal{A} : \mathfrak{m} \subseteq M_\mathcal{A}$ follows as in Theorem~\ref{thm:HZsaturated} because $F_\mathcal{A} \subseteq M_\mathcal{A}$ and hence, $F_\mathcal{A} : \mathfrak{m} \subseteq 
M_\mathcal{A} : \mathfrak{m} = M_\mathcal{A}$. The other containment will follow by showing $H^2_\mathcal{A}, H^3_\mathcal{A} \subseteq F_\mathcal{A} : \mathfrak{m}$. For general camera arrangements with $n$ cameras, recall that $F^2_\mathcal{A}$ (resp. $F^3_\mathcal{A}$) is the ideal generated by all $4 \times 4$ minors of $\mathcal{A}^F(p)$ that involve only two (resp. three) cameras. By Lemma~\ref{lem:faugerasbumping}(a), for any multilinear monomial $(\prod_{m=1}^n w_m)$ and any bifocal $b_{ij}$, $( \prod w_m ) b_{ij} \in (f)$ for some Faugeras minor $f \in F^2_\mathcal{A}$, hence $H^2_\mathcal{A} \subseteq F_\mathcal{A} : \mathfrak{m}$. We address the trifocals in two cases. First consider the case when the two rows eliminated from $\mathcal{A}_{\{i,j,k\}}(p)$ to form a trifocal $t \in H^3_{\{i,j,k\}}$ come from the same camera, say without loss of generality, from camera $i$. In this case, $t = w_i b_{jk}$ for some $w_i$, and Lemma~\ref{lem:faugerasbumping}(a) again implies $t \in F_\mathcal{A} : \mathfrak{m}$. For the case when the two rows from $\mathcal{A}_{\{i,j,k\}}(p)$ to form $t \in  H^3_{\{i,j,k\}}$ come from different cameras, Lemma~\ref{lem:faugerasbumping}(b) implies that, for any $(\prod w_m)$, $(\prod w_m)t\in (f)$ for some $f \in F^3_\mathcal{A}$. We conclude that $H^3_\mathcal{A} \subseteq F_\mathcal{A} : \mathfrak{m}$, as desired.
\end{proof}

\section{The Bifocal Ideal}
\label{sec:bifocal ideal}
We saw in Theorem~\ref{thm:BA+TA=JA} that the bifocals and trifocals together generate the multiview ideal when the camera foci are pairwise distinct. 
In this section, we investigate how imposing further conditions on the cameras can lead to an even simpler description of the multiview ideal. Heyden and \AA str\"{o}m~\cite{HA97} and Trager \etal~\cite{THP15} show that when the camera foci are not all on a plane, the bifocals are necessary and sufficient to cut out the multiview variety. There has also been work to further reduce this description by considering the minimal number of bifocals needed (\cite{HA97},~\cite{TOP18}), though we will not address this question here. 
In this section, we focus on the ideal-theoretic relationship between the bifocal ideal $H^2_\mathcal{A}$ and the multiview ideal $M_\mathcal{A}$ when the camera foci are noncoplanar. 

To motivate our investigation, we start with some examples. We say that a camera arrangement $\mathcal{A}$ is coplanar, noncoplanar or collinear if their foci have the corresponding property.

\begin{example}
\label{ex:noncoplanar4}
Consider the four noncoplanar  translational camera arrangement $\mathcal{A}_1$ where $\mathbf{t}_1 = (0,0,0)$, $ \mathbf{t}_2  = (1,0,0)$, $\mathbf{t}_3 = (0,1,0)$, $\mathbf{t}_4 = (0,0,1)$. Eliminating the variables $q$ and $\la_i$ from the ideal $\<A_i q - \la_i p_i : i = 1, \dots, n\>$, we observe $M_{\mathcal{A}_1}$ occurs as a component in $H^2_{\mathcal{A}_1}$
\begin{align*}
&H^2_{\mathcal{A}_1} = M_{\mathcal{A}_1}  \cap  \langle x_2,y_2,z_2, x_1,x_3,x_4\rangle \\
& \cap \langle x_1,y_1,z_1, x_2 + y_2 + z_2, x_3 + y_3 + z_3,x_4 + y_4 + z_4\rangle    \\
& \cap \langle x_3, y_3,z_3, y_1,y_2,y_4\rangle \cap \langle x_4,y_4,z_4,z_1,z_2,z_3\rangle.
\end{align*}
\end{example}

\begin{example}
\label{ex:coplanar4}
Consider the four coplanar translational camera arrangement $\mathcal{A}_2$ where $\mathbf{t}_1 = (1,0,0)$, $ \mathbf{t}_2  = (0,1,0)$, $\mathbf{t}_3 = (0,0,1)$, $\mathbf{t}_4 = (1/3,1/3,1/3)$. We observe that $H_{\mathcal{A}_2}^2 = M_{\mathcal{A}_2} \cap C$ where
\[
C = \langle{x}_{4}+{y}_{4}+{z}_{4},{x}_{3}+{y}_{3}+{z}_{3},{x}_{2}+{y}_{2}+{z}_{2},{x}_{1}+{y}_{1}+{z}_{1}\rangle.
\]
\end{example}

In Example~\ref{ex:noncoplanar4}, each extra component of $H^2_{\mathcal{A}_1}$ contains an irrelevant ideal $\mathfrak{m}_i$ and hence does not contribute to $\mathbf{V}(H^2_{\mathcal{A}_1})$. Saturating the bifocal ideal $H^2_{\mathcal{A}_1}$ with respect to the full irrelevant ideal $\mathfrak{m}$ removes these components. We will prove that this is always true when camera foci are noncoplanar. We begin by proving a series of three lemmas.

\begin{lemma}
\label{lemma:atob}
Suppose $\mathcal{A}$ is an arrangement of $n \geq 4$ cameras with pairwise distinct foci. Then $\mathcal{A}$ is noncoplanar $\implies$ $H_\mathcal{A}^n\subseteq H_\mathcal{A}^2$.
\end{lemma}

\begin{proof}
{$\mathbf{ n=4,5,6}$}. If $\mathcal{A}$ is noncoplanar, then there is some subset of four cameras that is noncoplanar. Order the cameras in $\mathcal{A}$ so that these are the cameras $A_1, \dots, A_4$. By a change of coordinates on $\P^3$, we can send the foci of the cameras $A_1, \dots, A_4$ to the foci of the cameras in $\mathcal{A}_1$ from Example~\ref{ex:noncoplanar4}. Then, by Lemma \ref{lem:kfocal ideals map}, applying $\P^2$ coordinate changes using some $\mathcal{G} \in (GL_3)^n$, we can assume that $\mathcal{A}$ is an arrangement of translational cameras. These transformations fix the first four cameras, and 
we think of the cameras $A_i$ for $i\ge 5$ as variable, represented symbolically by their translations, and the implication  can confirmed by direct calculation in Macaulay2.

{$\mathbf{ n=7}$.} In this case, the full computation is too expensive. To make the computation feasible, we split the proof into two cases, depending on whether the arrangement has five collinear cameras or not.

{\bf Case I:} If a noncoplanar arrangement of seven cameras has at most four collinear cameras, then every four camera subset can be augmented with two additional cameras to get a noncoplanar arrangement of six cameras. Thus every 7-focal of such an arrangement, which looks like $w_i w_j w_k q$ for some quadrifocal $q$, has the form of a 6-focal from a noncoplanar arrangement, say $w_i w_j q$, multiplied by a coordinate $w_k$.
The $n=6$ case shows that $w_i w_j q$ is generated by 2-focals, hence $w_i w_j w_k q$ is generated by 2-focals.

{\bf Case II:} We now consider the case of noncoplanar seven camera arrangements in which five cameras are collinear.
In this case, by a proper choice of camera ordering and $\P^3$ coordinate change, we can assume the translations of $A_5, A_6, A_7$ are of the form $\mathbf{t}_5 = (\la_5, 0,0)^\top,\mathbf{t}_5 = (\la_6, 0,0)^\top,\mathbf{t}_5 = (\la_7, 0,0)^\top $ where the $\lambda_i$ are symbolic. 
This makes $A_1, A_2, A_5,A_6,A_7$ collinear. The choice to take the line that the cameras lie on to be the $x$ axis is arbitrary, but can be made without loss of generality. This arrangement is now described by few enough variables to enable a direct computation showing that  $H_\mathcal{A}^7 \subseteq H_\mathcal{A}^2$. 

$\mathbf{n \ge 8}$. Now suppose $n \geq 8$ and $f$ is an 
$n$-focal of $\mathcal{A}$. 
Recall that $f$ involves all $n$ cameras but at most four cameras can contribute two rows to the matrix whose determinant is $f$. At one extreme, these four cameras maybe $A_1, \ldots, A_4$ and at the other extreme they might be four cameras different from the first four, which we call $A_5, \ldots, A_8$. Thus the $n$-focal $f \in H_\mathcal{A}^n$ is a monomial multiple of a 8-focal $g = m q$ of  $\{A_1, \dots, A_4, A_5, \dots, A_8\}$ where 
where $q$ is a quadrifocal and $m$ is a monomial.

If the four cameras contributing to $q$ involve $A_1, \dots, A_4$, then $g$ is a multiple of a 7-focal from noncoplanar cameras. On the other hand, if $q \in H_{A_5, \dots, A_8}^4$, then $q$ can be generated by the trifocals of $A_5, \ldots, A_8$ by Lemma \ref{lem:Q inside T for translational cameras}:
\[
g = m \left(\sum_{t_i \in H_{A_5, \dots, A_8}^3} h_i t_i \right) =  \sum_{t_i \in H_{A_5, \dots, A_8}^3} h_i (m t_i).
\]
In particular, this shows that $g$ can be generated from 7-focals, $m t_i$. These come from noncoplanar seven camera arrangements because $A_1, \dots, A_4$ are noncoplanar. In either case, we know that such 7-focals can be generated by 2-focals, hence $g \in H_\mathcal{A}^2$. It follows that $f \in H_\mathcal{A}^2$, as desired.
\end{proof}

\begin{lemma}
\label{lemma:btoc}

Suppose $\mathcal{A}$ is an arrangement of $n \geq 4$ cameras with pairwise distinct foci. Then $H_\mathcal{A}^n\subseteq  H_\mathcal{A}^2 \implies M_\mathcal{A} = H_\mathcal{A}^2 : \mathfrak{m}$.
\end{lemma}

\begin{proof}
If $f \in H^2_\mathcal{A} : \mathfrak{m}$, then $f(\prod z_i) \in H^2_\mathcal{A} \subseteq M_\mathcal{A}$, vanishes on $\mathbf{V}(M_\mathcal{A})$. Since $M_\mathcal{A}$ is prime and does not contain any monomials, $f \in M_\mathcal{A}$. Therefore, $H^2_\mathcal{A} : \mathfrak{m} \subseteq M_\mathcal{A}$.
For the other containment, by Theorem \ref{thm:BA+TA=JA}, it suffices to show that $H^2_\mathcal{A}$ and $H^3_\mathcal{A}$ are contained in $H^2_\mathcal{A} : \mathfrak{m}$. It is clear that $H^2_\mathcal{A} \subseteq H^2_A : \mathfrak{m}$. By Lemma~\ref{lem:bumping}, 
multiplying any $t \in H^3_\mathcal{A}$ by a generator $\prod w_i $ of $\mathfrak{m}$ yields a monomial multiple of an $n$-focal. By assumption, this $n$-focal lies in $H_\mathcal{A}^2$. Thus, 
$t  \in  H^2_A: \mathfrak{m}$ and $M_\mathcal{A} \subseteq H^2_A: \mathfrak{m}$.
\end{proof}

\begin{lemma}
\label{lemma:ctoa}
Suppose $\mathcal{A}$ is an arrangement of $n \geq 4$ cameras with pairwise distinct foci. Then $M_\mathcal{A} = H_\mathcal{A}^2 : \mathfrak{m} \implies \mathcal{A}$ is noncoplanar.
\end{lemma}
\begin{proof}
We prove the contrapositive, namely that if $\mathcal{A}$ is coplanar then 
$M_\mathcal{A} \neq H^2_A: \mathfrak{m}$. We will  construct a point $\mathbf{p} \in \mathbf{V}(H^2_\mathcal{A} : \mathfrak{m}) \minus \mathbf{V}(M_\mathcal{A})$, from which the result will follow. 

Let $\mathbf{n} \in \mathbb{P}^3$ be the normal vector of a plane containing the foci of the cameras in $\mathcal{A}$. If the foci are not collinear then $\mathbf{n}$ is unique, otherwise we choose any plane containing the foci and its normal $\mathbf{n}$.  Let $l_i \subseteq \P^2$ denote the image of the plane $\mathbf{n}^\perp$ in camera $i$, and let $\mathbf{e}_{i,j}$ denote the image of the focal point of camera $j$ in image $i$. Then $\mathbf{e}_{i,j} \in l_i$ since the focal point of camera $j$ lies in $\mathbf{n}^\perp$. Choose $\mathbf{p}_1 \in l_1 \minus \{\mathbf{e}_{1,2} , \mathbf{e}_{1,3}\}$ and $\mathbf{p}_2 \in l_2  \minus\{ \mathbf{e}_{2,1} , \mathbf{e}_{2,3}\}$. Then there is a unique world point $\mathbf{q}$ on $\mathbf{n}^\perp$ whose images in cameras $1$ and $2$ are $\mathbf{p}_1$ and $\mathbf{p}_2$. Let $\tilde{\mathbf{p}}_3 \in l_3$ be the (unique) image of $\mathbf{q}$ in camera $3$. Then $\mathbf{p}_1, \mathbf{p}_2, \tilde{\mathbf{p}}_3$ satisfy trifocal constraints. Choose $\mathbf{p}_3 \in l_3 \minus \{\tilde{\mathbf{p}}_3\}$ and some $\mathbf{p}_i \in l_i$ for $i\ge 4$. By construction, $\mathbf{p} \notin \mathbf{V}(M_\mathcal{A})$. Since the cameras are coplanar, the epipolar plane given by $\mathbf{q}$ and any two cameras $i$ and $j$ is $\mathbf{n}^\perp$ for any pair $i,j$. By choosing $\mathbf{p}_i \in l_i$ for all $i$, we force every bifocal polynomial to vanish on $\mathbf{p}$. Therefore by construction, $\mathbf{p} \in \mathbf{V}(H_\mathcal{A}^2) \minus \mathbf{V}(M_\mathcal{A})$, but since $\mathbf{V}( H^2_\mathcal{A}) = \mathbf{V}( H^2_\mathcal{A} : \mathfrak{m})$, we conclude that $H^2_\mathcal{A} : \mathfrak{m}  \neq M_\mathcal{A}$.
\end{proof}

\noindent Together, Lemmas~\ref{lemma:atob}, \ref{lemma:btoc}, \ref{lemma:ctoa} imply the following theorem.
\begin{theorem}\label{thm:bifocals saturated}
Suppose $\mathcal{A}$ is an arrangement of $n \geq 4$ cameras with pairwise distinct foci. Then the following are equivalent.
\begin{enumerate}[label=(\alph*)]
    \item $\mathcal{A}$ is noncoplanar.
    \item $H_\mathcal{A}^n\subseteq H_\mathcal{A}^2$.
    \item $M_\mathcal{A} = H_\mathcal{A}^2 : \mathfrak{m}$.
\end{enumerate}
\end{theorem}

We now make some observations about Theorem~\ref{thm:bifocals saturated}.

Theorem 6.1 in~\cite{HA97} observes that $\mathbf{V}(H_\mathcal{A}^2) = \mathbf{V}(M_\mathcal{A})$ for noncoplanar $\mathcal{A}$ while Proposition 5 (2) in \cite{THP15} further shows that $\mathbf{V}(H_\mathcal{A}^2) = \mathbf{V}(M_\mathcal{A})$ is equivalent to the foci of $\mathcal{A}$ being noncoplanar. Our Theorem~\ref{thm:bifocals saturated} proves the analogous ideal statement, namely that noncoplanarity of foci is equivalent to $M_\mathcal{A} = H^2_\mathcal{A} \,:\, \mathfrak{m}$.

Example~\ref{ex:coplanar4} shows how Theorem ~\ref{thm:bifocals saturated} fails when $\mathcal{A}$ is coplanar. The bifocal ideal $H^2_{\mathcal{A}_2}$ contains the component $\langle x_1 + y_1 + z_1, x_2 + y_2 + z_2, x_3 + y_3 + z_3, x_4 + y_4 + z_4\rangle $, which cannot be removed by saturating with respect to $\mathfrak{m}$. Its variety cuts out the 
projections of the plane containing the foci of $\mathcal{A}_2$ in each camera image. This 
plane in $\P^3$ has normal vector $(1,1,1,-1)$. The following example shows that further degeneracy occurs when camera foci are collinear.

\begin{example}
\label{ex:collinear4}
Consider the four collinear translational camera arrangement $\mathcal{A}_3$ where $\mathbf{t}_1 = (0,0,0)$, $ \mathbf{t}_2  = (1,0,0)$, $\mathbf{t}_3 = (2,0,0)$, $\mathbf{t}_4 = (3,0,0)$. Here, $H_{\mathcal{A}_3}^2 \subseteq M_{\mathcal{A}_3}$, but both ideals are prime, so $M_{\mathcal{A}_3}$ cannot occur as a component of $H_{\mathcal{A}_3}^2$. In addition, the dimension of   $H_{\mathcal{A}_3}^2$ is one larger than that of $M_{\mathcal{A}_3}$. This is explained by the fact that there is an entire one-dimensional family of planes that contains the camera centers of $\mathcal{A}_3$.
\end{example}

As seen in the above examples and discussion, the relation between $H_\mathcal{A}^2$ and  $M_\mathcal{A}$ can be complicated when camera centers are coplanar or collinear. Determining the exact relationship between ideals in these degenerate settings would be an interesting problem for the future.

In Theorem~\ref{thm:bifocals saturated} we showed that 
when cameras are noncoplanar, the $n$-focal ideal becomes 
a subset of the $2$-focal ideal. We now give an example 
to show that this containment need not hold for 
$H^k_\mathcal{A}$ where $n > k > 2$. The construction relies on having three of five cameras being collinear.

\begin{example}
Consider the five translational camera arrangement $\mathcal{B}$ with $\mathbf{t}_1 = (0,0,0), \mathbf{t}_2 = (0,0,1) , \mathbf{t}_3 = (0,0,2), \mathbf{t}_4 = (0,1,0), \mathbf{t}_5 = (0,0,1)$. Theorem \ref{thm:bifocals saturated} shows that $H_\mathcal{B}^5 \subseteq H_\mathcal{B}^2$ since $\mathcal{B}$ is noncoplanar. However the following  trifocal from $B_1, B_2, B_3$, 
\[
t = -{x}_{1}{y}_{2}{y}_{3}+2\,{x}_{2}{y}_{1}{y}_{3}-{x}_{3}{y}_{1}{y}_{2}
\]
is not in $H_\mathcal{B}^2$. Similarly, the quadrifocal, 
\[
q = x_4 t = -{x}_{1}{x}_{4}{y}_{2}{y}_{3}+2\,{x}_{2}{x}_{4}{y}_{1}{y}_{3}-{x}_{3}{x}_{4}{y}_{1}{y}_{2},
\]
from cameras $B_1, B_2, B_3, B_4$ is not in $H_\mathcal{B}^2$. 
\end{example}

\section{Finite Images}
\label{sec:finite images}

%The real points in this affine patch can be identified with the affine variety in $\R^{2n}$ consisting of all real two-dimensional images taken by the cameras in $\mathcal{A}$. 

The results of the previous sections have important practical consequences when we restrict attention to the 
set of all finite images, that is to all $(\mathbf{p_1}, \dots, \mathbf{p_n}) \in \mathbf{V}(M_\mathcal{A})$ with $z_i \neq 0$ for all $i$. 
The vanishing ideal of this affine patch is obtained by dehomogenizing $M_\mathcal{A}$ with respect to the variables $z_i$ from each image plane. We call this the {\em affine multiview ideal} of $\mathcal{A}$ and denote it $\pi(M_\mathcal{A})$, where $\pi : \C[x_i,y_i,z_i] \to \C[x_i,y_i]$ is the map setting each $z_i$ to 1. From Theorem~\ref{thm:BA+TA=JA}, we see that $\pi(M_\mathcal{A})$ is generated by dehomogenized bifocals and dehomogenized trifocals when the foci of $\mathcal{A}$ are pairwise distinct.

\begin{corollary}\label{cor:dehomogenizingMA}
If $\mathcal{A}$ is a camera arrangement with pairwise distinct foci, then 
$\pi(M_\mathcal{A}) = \pi(H_\mathcal{A}^2) + \pi(H_\mathcal{A}^3)$.
\end{corollary}

Using the following fact about dehomogenizing colon ideals,  the results of Section~\ref{sec:determinantal ideals} yield a nice relation among $\pi(H_\mathcal{A}^n), \pi(F_\mathcal{A}), \pi(Y_\mathcal{A}) $, and the affine multiview ideal, $\pi(M_\mathcal{A})$.

\begin{lemma}
\label{lem:dehomogenizecolon}
For ideals $I, J \subset \C[x_i,y_i,z_i] $,  $\pi (I:J) = \pi(I) : \pi(J)$. 
\end{lemma}

\begin{proof}
If $f \in \pi(I:J)$, then $f = \pi(g)$ for some $g$ which satisfies $g h \in I$ for all $h \in J$. Therefore $f\pi(h) = \pi(g) \pi(h) = \pi(gh) \in \pi(I)$ for any $h \in J$, proving $f \in \pi(I):\pi(J)$. If $f \in \pi(I):\pi(J)$, then for any $h \in J$, $f \pi(h) \in \pi(I)$, \ie, there exists $g \in I$ such that $f \pi(h) = \pi(g)$. Denote the homogenization of $f$ with respect to $z_1, \dots, z_n$ by $\tilde{f}$. We claim that $\tilde{f} \in I:J$. Indeed for any $h \in J$, $ \pi( \tilde{f} h) =\pi(\tilde{f})\pi(h) =  f \pi(h) = \pi(g)$ for some  $g \in I$. Homogenizing both sides, we get $\tilde{f} h = g \in I$, and we conclude that $\pi(I):\pi(J) \subseteq \pi(I:J)$
\end{proof}

\begin{corollary}\label{cor:dehomogenizing}
If $\mathcal{A}$ is a camera arrangement with pairwise distinct foci, then $\pi(M_\mathcal{A}) = \pi(H_\mathcal{A}^n) = \pi(F_\mathcal{A}) = \pi(\sqrt{Y_\mathcal{A} }) $.
\end{corollary}

\begin{proof}
Lemma~\ref{lem:dehomogenizecolon} implies that $\pi(I : \mathfrak{m}) = \pi(I) : (1) = \pi(I)$ for any ideal $I$. Dehomogenizing Theorems~\ref{thm:HZsaturated},~\ref{thm:FAsaturated}, and~\ref{thm:radicalsaturated}, each equality follows.
\end{proof}

Observe that the last equality in Corollary~\ref{cor:dehomogenizing} requires $A_1 = [I\, \, 0]$. Geometrically, Corollary~\ref{cor:dehomogenizing} shows that while the homogenous ideals $H_\mathcal{A}^n, F_\mathcal{A}, Y_\mathcal{A}$, and $M_\mathcal{A}$ do not coincide, they are the same away from the origin in each image plane. In particular, this is the case on the affine patch $\{\mathbf{p} \in \P^{2n} : z_1 = \dots = z_n =  1\}$ corresponding to finite image data.

 Using Theorem~\ref{thm:bifocals saturated} we see that, when $\mathcal{A}$ is noncoplanar, the dehomogenized bifocals alone suffice to generate the affine multiview ideal $\pi(M_\mathcal{A})$.  

\begin{corollary}\label{cor:dehomogenizingnoncoplanar}
Suppose $\mathcal{A}$ is a noncoplanar camera arrangement with pairwise distinct foci.
Then 
\begin{align*}
\pi(M_\mathcal{A}) = \pi(H_\mathcal{A}^2).
\end{align*}
\end{corollary}

\begin{proof}
Dehomogenizing the result of Theorem~\ref{thm:bifocals saturated}, we get $\pi(M_\mathcal{A}) = \pi(H_\mathcal{A}^2 : \mathfrak{m}) =  \pi(H_\mathcal{A}^2) : \pi(\mathfrak{m})  = \pi (H_\mathcal{A}^2).$
\end{proof}

Corollary \ref{cor:dehomogenizingnoncoplanar} shows that $\pi(M_\mathcal{A})$ is generated by quadratics whenever $\mathcal{A}$ satisfies the noncoplanarity assumption. 
This observation was used in \cite{AAT12} to create a semidefinite programming relaxation of the triangulation problem which is can be seen as minimizing Euclidean distance from an observed noisy data point to the affine multiview variety. It was shown that when the noise is small, the semidefinite relaxation solves triangulation. Of course, Corollary~\ref{cor:dehomogenizing} needs the foci of the cameras to be noncoplanar and indeed, the experiments in \cite{AAT12} show that the quality of the semidefinite programming solution deteriorates as the foci become coplanar and then collinear. 

Geometrically, we can understand how the quality of the relaxation deteriorates because the bifocal ideal cuts out more than the multiview variety for coplanar arrangements. In the coplanar case, the bifocal ideal cuts out the image of the plane that contains the camera centers. These points are not the images of true 3D points. It is therefore possible that the nearest point problem yields a spurious solution on this extra component. Similarly, in the collinear case, the bifocal ideal cuts out a strictly larger variety than just the multiview variety. In this case, the dimension of the vanishing set of the bifocal ideal is one larger than the multiview variety. 

\section{Summary}
\label{sec:summary}
The multiview variety is a foundational geometric object in multiview geometry and understanding its vanishing ideal  $M_\mathcal{A}$ precisely is important for any algebraic algorithm that solves problems on this variety.
There have been many partial results about the algebraic structure of the multiview variety. The aim of our paper is to put them all into a unified algebraic setting and give a complete  description of  $M_\mathcal{A}$.

Our main result is that when the foci of the cameras are pairwise distinct, $M_\mathcal{A}$ is generated by the bifocal and trifocal polynomials of $\mathcal{A}$ 
(Theorem~\ref{thm:BA+TA=JA}). 
The proof requires an understanding of the behavior of coordinate changes 
on $k$-focal ideals (Lemma~\ref{lem:kfocal ideals map}), and translational cameras 
(Lemma~\ref{lem:Q inside T for translational cameras}).
The main result holds for Euclidean cameras as well (Corollary~\ref{cor:Euclidean cameras}). We also give an example to illustrate that the 
assumption of distinct foci cannot be relaxed for this result to hold (Example~\ref{ex:coincident foci2}).

Next we study three sets of polynomials that have been proposed to cut out the multiview variety, by Heyden-{\AA}str{\"o}m, Faugeras and Ma et. al. respectively. We show that the ideals generated by these polynomials 
are all properly contained in $M_\mathcal{A}$. We establish the exact algebraic relationships between the above ideals and $M_\mathcal{A}$ 
(Theorems \ref{thm:radicalsaturated}, \ref{thm:HZsaturated} and \ref{thm:FAsaturated}). 

We then prove that if the camera foci are assumed to be noncoplanar, then in fact $M_\mathcal{A}$ is the saturation of the bifocal ideal by the irrelevant ideal (Theorem~\ref{thm:bifocals saturated}). In this situation the $n$-focal ideal is a subset of the bifocal ideal.

Finally we prove that the dehomogenization of the ideals by Heyden-{\AA}str{\"o}m, Faugeras and Ma et. al. all agree with the dehomogenization of $M_\mathcal{A}$ 
(Corollary~\ref{cor:dehomogenizing}). Similarly, under noncoplanarity of foci, the bifocal ideal also has the same dehomogenization (Corollary~\ref{cor:dehomogenizingnoncoplanar}).
This means that all of these ideals cut out the space of finite images. 

\section{Acknowledgements}
We wish to thank the referees of this paper for their careful reading and suggestions. In particular, their comments helped fill a gap in the proof of the main 
theorem of Section~5.

Andrew Pryhuber and Rekha R. Thomas acknowledge support from the U.S. National Science Foundation through the grant DMS-1719538. 

\bibliography{references}
\bibliographystyle{siam}

% \begin{IEEEbiography}[{\includegraphics[width=1in,clip,keepaspectratio]{SameerAgarwal.jpg}}]{Sameer Agarwal} received his Ph.D. in Computer Science from University of California, San Diego in 2006. He is currently an engineer at Google and an Affiliate Professor at the University of Washington. His research interests are in computer vision, optimization and algebra.
% \end{IEEEbiography}
% \begin{IEEEbiography}[{\includegraphics[width=1in,clip,keepaspectratio]{AndrewPryhuber.jpg}}] {Andrew Pryhuber}
% received his BA in mathematics from Bowdoin College. He is currently a doctoral student in the Department of Mathematics at the University of Washington. His research interests include computational algebraic geometry, computer vision, and optimization.
% \end{IEEEbiography}
% \begin{IEEEbiography}[{\includegraphics[width=1in,clip,keepaspectratio]{RekhaThomas.jpg}}]{Rekha R. Thomas}
% received her Ph.D. in Operations Research from Cornell University in 1994. She is currently a Professor of Mathematics at the University of Washington. Her research interests are in optimization \& applied algebraic geometry.
% \end{IEEEbiography}

%\appendices
\section*{Appendix A: Multigraded Projective Nullstellensatz}
 \label{sec:appendixA}
 
In this appendix, we state and prove the projective Nullstellensatz in our multigraded setting, which we use to prove Theorem~\ref{thm:radicalsaturated} in Section~\ref{sec:determinantal ideals}. Let $I \subseteq\C[p_1, \dots, p_n]$ be homogeneous with respect to the $\Z^n$-grading $\deg w_i = \mathbf{e}_i$. To be clear about projective versus affine varieties, we define
$\mathbf{V}_\mathbb{P}(I ) := \mathbf{V}(I) = \{\mathbf{p} \in (\P^2)^n : f(\mathbf{p}) = 0 \text{ for all } f \in I \}$, and 
for a set  $S\subseteq (\P^2)^n$, we define 
% $\mathbf{V}_\mathbb{P}(-)$ and $\mathbf{I}_\mathbb{P}(-)$:
\begin{align*}
%\mathbf{V}_\mathbb{P}(I ) &= \{\mathbf{p} \in (\P^2)^n : f(\mathbf{p}) = 0 \text{ for all } f \in I \} \\
\mathbf{I}_\mathbb{P}(S) &= \{f \in \mathfrak{m} : f(\mathbf{p}) = 0 \text{ for all } \mathbf{p} \in S\}.
\end{align*}
We say that $\mathbf{V}_\mathbb{P}(I) $ is the projective vanishing set of $I$ in $(\P^2)^n$ and $\mathbf{I}_\mathbb{P}(S)$ is the largest homogeneous ideal vanishing on $S$ contained in $\mathfrak{m}$. While we force $\mathbf{I}_\mathbb{P}(S) \subseteq \mathfrak{m}$, it also makes sense to consider the largest homogeneous ideal vanishing on $S$ without intersecting with $\mathfrak{m}$. As before we denote this ideal by $\mathbf{I}(S)$, and notice that $\mathbf{I}_\mathbb{P}(S) = \mathbf{I}(S) \cap \mathfrak{m}$. In the usual grading on $\C[p_1, \ldots, p_n]$, a vanishing ideal $\mathbf{I}(S)$ is homogeneous in the usual sense which means that it is contained in the usual irrelevant ideal $\langle x_1,y_1,z_1, \ldots, x_n,y_n,z_n \rangle$. Under the multi-grading, $\mathbf{I}_\mathbb{P}(S)$ is required to be in the corresponding irrelevant ideal $\mathfrak{m}$.
We will use the following variant of the Nullstellensatz.
\begin{lemma}
\label{lem:nullstellensatz}
For any homogeneous ideal $I \subseteq \C[p_1, \dots, p_n]$ such that $I \subseteq \mathfrak{m}$, $\mathbf{I}_\mathbb{P}(\mathbf{V}_\mathbb{P}(I)) = \sqrt{I}. $
\end{lemma}

\begin{proof}
Define the affine operations 
\begin{align*}
\mathbf{V}_\mathbb{A}(I) &= \{  \mathbf{p} \in (\A^3)^n : f(\mathbf{p}) = 0 \text{ for all } f \in I\} \\
\mathbf{I}_\mathbb{A}(S) &= \{f \in \C[p_1, \dots, p_n]: f(\mathbf{p}) = 0 \text{ for all } \mathbf{p} \in S\}
\end{align*}
where we treat $S$ as a subset of $(\mathbb{A}^3)^n$. We will use the affine version of the Nullstellensatz on the cone over $V:= \mathbf{V}_\mathbb{P}(I)$, \ie, the set $C_V = \mathbf{V}_\mathbb{A}(I) \subseteq (\A^3)^n$. We claim that 
\begin{align}
\label{eq:cone}
\mathbf{I}_\mathbb{A}(C_V) = \mathbf{I}_\mathbb{P}(V).
\end{align}

First suppose $f \in \mathbf{I}_\mathbb{A}(C_V)$. Given $\mathbf{p} = (\mathbf{p}_1, \dots, \mathbf{p}_n) \in V$, all homogeneous coordinates of $\mathbf{p}$, represented by scalings $(\la_1 \mathbf{p}_1, \dots , \la_n\mathbf{p}_n)$, lie in $C_V$, so $f$ vanishes for all homogeneous coordinates of $\mathbf{p}$. This means that the homogeneous components $f_{i_1, \dots, i_n}$ of $f$, consisting of all terms with multidegree $(i_1, \dots, i_n)$, vanish at $\mathbf{p}$, so 
$f \in \mathbf{I}(V)$, hence $\mathbf{I}_\mathbb{A}(C_V) \subseteq \mathbf{I}(V)$. By the Nullstellensatz in $(\A^3)^n$, $\mathbf{I}_\mathbb{A}(C_V) = \mathbf{I}_\mathbb{A}(\mathbf{V}_\mathbb{A}(I) ) = \sqrt{I}$, and by the assumption that $I \subseteq \mathfrak{m}$, $\sqrt{I} \subseteq \sqrt{\mathfrak{m}} = \mathfrak{m}$. This shows that $ \mathbf{I}_\mathbb{A}(C_V) \subseteq  \mathbf{I}(V) \cap \mathfrak{m} =  \mathbf{I}_\mathbb{P}(V)$.

Conversely, suppose $f \in \mathbf{I}_{\mathbb{P}}(V)$. Since any point $\mathbf{p}$ of $C_V$ such that $\mathbf{p}_i \neq 0$ for all $i$ gives homogeneous coordinates for a point in $V$, it follows that $f$ vanishes on $C_V \minus \bigcup_{i=1}^n \A^3 \times \dots \times \{0\}_i \times \dots \times \A^3$. We need to show that $f$ vanishes on each of the sets $\A^3 \times \dots \times \{0\}_i \times \dots \times \A^3$. Since $f \subseteq \mathfrak{m}$, it has strictly positive multidegree, and every monomial in $f$ contains at least one coordinate from each copy of $\A^3$. Setting all 3 coordinates to zero in any $\A^3$ forces $f$ to be zero, so we conclude that $f \in \mathbf{I}_\mathbb{A}(C_V)$. Finally, from (\ref{eq:cone}), we conclude
\[
\sqrt{I} = \mathbf{I}_\mathbb{A}(\mathbf{V}_\mathbb{A}(I)) = \mathbf{I}_\mathbb{A}(C_V) = \mathbf{I}_\mathbb{P}(V) =   \mathbf{I}_\mathbb{P}(\mathbf{V}_\mathbb{P}(I)).
\]
\end{proof}

\begin{corollary}
\label{cor:nullstellensatz}
For any homogeneous ideal $I \subseteq \C[p_1, \dots, p_n]$,  $\mathbf{I}_\mathbb{P}(\mathbf{V}_\mathbb{P}(I)) = \sqrt{I} \cap \mathfrak{m}$.
\end{corollary}

\begin{proof}

Observe that 
\[\mathbf{V}_\mathbb{P}(I \cap \mathfrak{m}) = \mathbf{V}_\mathbb{P}(I) \cup \mathbf{V}_\mathbb{P}(\mathfrak{m}) =\mathbf{V}_\mathbb{P}(I)
\] 
and 
\[
\sqrt{I \cap \mathfrak{m}} = \sqrt{I} \cap \sqrt{\mathfrak{m}} = \sqrt{I} \cap \mathfrak{m} \subseteq \mathfrak{m}
\]
Therefore by Lemma~\ref{lem:nullstellensatz}, $\mathbf{I}_\mathbb{P}(\mathbf{V}_\mathbb{P}(I)) = \sqrt{I} \cap \mathfrak{m}$.
\end{proof}

\begin{corollary}
\label{cor:same intersection}
For any $\mathcal{A}$ with pairwise distinct foci,
\[
M_\mathcal{A} \cap \mathfrak{m} =\sqrt{H^n_\mathcal{A}} \cap \mathfrak{m} = \sqrt{F_\mathcal{A}} \cap \mathfrak{m} = \sqrt{Y_\mathcal{A}} \cap \mathfrak{m}.
\]
\end{corollary}

\begin{proof}
We have already shown in Section~\ref{sec:determinantal ideals} that $\mathbf{V}_\mathbb{P}(H^n_\mathcal{A}) = \mathbf{V}_\mathbb{P}(F_\mathcal{A}) = \mathbf{V}_\mathbb{P}(Y_\mathcal{A}) = \mathbf{V}_\P(M_\mathcal{A}) $. Since $M_\mathcal{A}$ is radical, the result follows by Corollary~\ref{cor:nullstellensatz}.
\end{proof}

We can now prove Theorem~\ref{thm:radicalsaturated}, restated here, from the main body of the paper.

\begin{theorem}
%\label{thm:radicalsaturated}
For any $\mathcal{A}$ with pairwise distinct foci,
\begin{enumerate}[label=\alph*)]
\item $\sqrt{H^n_\mathcal{A}} : \mathfrak{m} = M_\mathcal{A}$
\item $\sqrt{F_\mathcal{A}} : \mathfrak{m} = M_\mathcal{A}$
\item $\sqrt{Y_\mathcal{A}} : \mathfrak{m} = M_\mathcal{A}$ when $A_1 = [I \; |\; 0]$
\end{enumerate}
\end{theorem}

\begin{proof}
Taking colon ideal with $\mathfrak{m}$, the desired result follows from Corollary~\ref{cor:same intersection} and the fact that $M_\mathcal{A} : \mathfrak{m} = M_\mathcal{A}$, which was proven in Theorem~\ref{thm:HZsaturated}.

\end{proof}
%\clearpage

\section*{Appendix B: Technical Proofs}
 \label{sec:appendixB}
 
 In this appendix, we elaborate on the technical details used to prove Theorem~\ref{thm:FAsaturated}. Recall that the nontrivial statement there was that bifocals and trifocals can be multiplied by any generator of $\mathfrak{m}$ to fall into $F_\mathcal{A}$. This requires understanding the $4 \times 4$ minors of $\mathcal{A}^F(p)$ for which we once again invoke the Cauchy-Binet formula and the observation that $\mathcal{A}^F(p) = P(p) \mathcal{A}$ from (\ref{eq:AFdecomposition}). 

First we characterize certain $4\times 4$ minors of $P(p)$. Let $p_{ij}$ denote the $j$th coordinate of $p_i$, \ie, $p_{i1} = x_i$, $p_{i2} = y_{i}$, and $p_{i3} = z_i$. Having the subscript (resp. superscript) $p_{ij}$ on $P(p)$ indicates eliminating from $P(p)$ the unique row (resp. column) of $[p_i]_\times$ that does not contain $p_{ij}$. On the other hand, having the subscript $p_{ij}$ on the matrices $\mathcal{A}$ and $\mathcal{A}(p)$ will stand for eliminating the unique row of the matrix containing $p_{ij}$.

We will only need to consider the $4\times 4$ minors of $P(p)$ when $n= 2$ and $n=3$. Let $R_i, C_i \subseteq \{p_{i1}, p_{i2}, p_{i3}\}$ denote collections of coordinates, and write $R = \bigcup_{i=1}^n R_i$, $C = \bigcup_{i=1}^n C_i$. 
When $n=2$, a $4 \times 4$ minor of $P(p)$ is $\det (P(p)_R^C)$ for some $R$, $C$ of size $|R| = |C| = 2$, and 
when $n=3$,  $|R| = |C| = 5$. Observe that if $|R_i| \neq |C_i|$ for any $i$, then the submatrix $P(p)_R^C$ has at least two linearly dependent rows or columns, yielding a zero minor. When $|R_i| = |C_i| $ for all $i$, $P(p)_R^C$ is block diagonal, so $\det (P(p)_R^C) = \prod_{i=1}^n \det  ( ([p_i]_\times)_{R_i}^{C_i}) $. 

\begin{lemma}
\label{lem:2Pminors}
Let $n=2$. The nonzero $4\times 4$ minors of $P(p)$ are determined by collections of coordinates $R, C$ with $|R_1| = |C_1| = |R_2| = |C_2| = 1$. For $R = \{p_{1j}, p_{2k} \}$ and $C = \{p_{1l}, p_{2m}\}$, the $4\times 4$ minor $\det (P(p)_R^C )$ is the monomial
\begin{align*}
\det (P(p)_{R}^C) = (-1)^{j+k+l+m} p_{1j}p_{2k}p_{1l}p_{2m}.
\end{align*}
\end{lemma}

\begin{proof}
As noted above, if $|R_i| \neq |C_i|$ for either $i$, then $\det (P(p)_R^C) = 0$, whereas if $|R_i| = |C_i| = 2$ for either $i$, then $P(p)_R^C$ has a rank 2 block on its diagonal, hence $\det(P(p)_R^C) = 0$, proving the first statement. For $R = \{p_{1j}, p_{2k} \}$ and $C = \{p_{1l}, p_{2m}\}$, the $4\times 4$ minor $\det P(p)_R^C$ is
\begin{align*}
\det P(p)_R^C &= \det ( ([p_1]_\times)_{R_1}^{C_1}  \det (  ([p_2]_\times)_{R_2}^{C_2} ) \\
			&= ((-1)^{j+l} p_{1j}p_{1l} ) ((-1)^{k+m} p_{2k}p_{2m}) \\
			&= (-1)^{j+k+l+m} p_{1j}p_{2k}p_{1l}p_{2m}.
\end{align*}
			
\end{proof}

\begin{lemma}
\label{lem:3Pminors}
Let $n=3$.  Suppose $|R_3| = |C_3| = 1$, and $|R_1|  = |C_1| = |R_2|= |C_2| = 2$. For $R_3 = \{p_{3j}\}, C_3 = \{p_{3k}\}$, the $4\times 4$ minor $\det (P(p)_R^C )$ is the monomial
\[
 \begin{cases}  (-1)^{j+k + l+m}  p_{3j}p_{3k} p_{1l} p_{2m}   & \text{ if } R_1 \neq C_1, R_2 \neq C_2 \\ 0 & \text{ otherwise.} \end{cases}
\]
where $p_{1l}$ is the coordinate common to $R_1$ and $C_1$ and $p_{2m}$ is the coordinate common to $R_2$ and $C_2$. 
\end{lemma}

\begin{proof}
When $R_i= C_i$ as sets for $i=1$ or $i=2$, then $([p_i]_\times)_{R_i}^{C_i} = 0$, hence $\det P(p)_R^C =  \prod_{i=1}^n \det ( ([p_i]_\times)_{R_i}^{C_i}) = 0$. On the other hand, when $R_1 \neq C_1$, $\det ( ([p_1]_\times)_{R_1}^{C_1} ) = (-1)^l p_{1l}$ where $p_{1l} = R_1\cap C_1$. Similarly $\det ( ([p_2]_\times)_{R_2}^{C_2} ) = (-1)^m p_{2m}$ where $p_{2m} = R_2\cap C_2$ when $R_2 \neq C_2$. 
\end{proof}

We now show that bifocals and trifocals can both be multiplied by any generator of $\mathfrak{m}$ to fall into $F_\mathcal{A}$.

%\fixme{Recall that $F^2_\mathcal{A}$ (resp. $F^3_\mathcal{A}$) is the ideal generated by all $4 \times 4$ minors of $\mathcal{A}^F(p)$ that involve only two (resp. three) cameras.  - Is this needed?}

 \begin{lemma}
%\label{lem:faugerasbumping}
\begin{enumerate}[label=\alph*)]
\item For $n=2$ cameras, and any monomial $p_{1j}p_{2k}$, there exists a $4 \times 4$ minor $f$ of $\mathcal{A}^F(p)$ such that $f = (-1)^{j+k} p_{1j}p_{2k} \det (\mathcal{A}(p) )$.
\item Let $n=3$ and $i_1,i_2,i_3$ be pairwise distinct. Then for any trifocal $ \det (\mathcal{A}(p)_{\{p_{i_1j_1}p_{i_2j_2} \}}) $ and any coordinate $p_{i_3k}$, there exists 
a $4 \times 4$ minor $f$ of $\mathcal{A}^F(p)$ such that 
$f = (-1)^k p_{i_3k} \det ( \mathcal{A}(p)_{\{p_{i_1j_1}p_{i_2j_2} \}})$.
\end{enumerate}
\end{lemma}

 \begin{proof}
(a) Fix some $p_{1j}p_{2k}$. Since $n=2$, $P(p)\mathcal{A}$ is a $6 \times 4$ matrix and we need to 
delete two rows to get a $4 \times 4$ minor. Using Lemma~\ref{lem:2Pminors} and Cauchy-Binet, the result follows from the computation below:
\begin{align*}
f   
&= \det \left( P(p)_{\{p_{1j}, p_{2k} \}} \mathcal{A} \right) \\
&= \sum_{|C| = 2} \det \left(P(p)_{\{p_{1j},p_{2k}\}}^C \right) \det \left(\mathcal{A}_C \right) \\
&= \sum_{|C_1| = |C_2| = 1} \det \left(P(p)_{\{p_{1j},p_{2k}\}}^C \right) \det \left(\mathcal{A}_C \right)  \\ 
&=  \sum_{1\le l,m\le 3}   \det \left(\left([p_1]_\times\right)_{\{p_{1j}\} }^{\{p_{1l}\} }\right)  \det \left(\left([p_2]_\times\right)_{\{p_{2k}\} }^{\{p_{2m}\}}\right)  
\times \\ & \qquad \qquad\quad 
\det \left(\mathcal{A}_{\{ p_{1l}, p_{2m} \} }\right) \\
&=  \sum_{1\le l,m\le 3}  (-1)^{j+ k + l+m} p_{1j}p_{2k}p_{1l} p_{2m} 
 \det \left(\mathcal{A}_{\{ p_{1l}, p_{2m} \} }\right) \\
&=  (-1)^{j+k} p_{1j}p_{2k} \sum_{1\le l,m\le 3}
(-1)^{l+m}p_{1l} p_{2m} \det \left(\mathcal{A}_{\{ p_{1l}, p_{2m} \} }\right)\\
 &=  (-1)^{j+k} p_{1j}p_{2k} \det \left(\mathcal{A}(p)\right).
\end{align*}
where the last equality follows from expanding the determinant of $\mathcal{A}(p)$ along the last two columns.

(b) Without loss of generality, let $i_1 = 1$, $i_2 = 2$, $i_3 =3$ and let $p_{3k}$ be arbitrary. For simplicity, suppose $j_1=j_2=1$. Therefore, we consider the trifocal $\det(\mathcal{A}(p)_{\{p_{11},p_{21}\}})$. Using Lemma~\ref{lem:3Pminors} and Cauchy-Binet, we expand $f = \det ( P(p)_R \mathcal{A} )$ where $R_1 = \{p_{12}, p_{13}\}$, $R_2 = \{p_{22}, p_{23}\}$, $R_3 = \{p_{3k}\}$ as follows:
\begin{align*}
f &= \det \left( P(p)_R \mathcal{A} \right) \\
 &= \sum_{C \,:\, |C_1| = |C_2| = 2, |C_3| = 1} \det \left(P(p)_{\{R_1,R_2,R_3\}}^C\right) \det( \mathcal{A}_C ) \\
 &= \sum_{|C_3| = 1} 
 \det \left( \left([p_3]_\times\right)_{R_3}^{C_3}\right)  \times \\
 & \qquad \qquad  \sum_{\substack{|C_1| = |C_2| = 2}}  
\Bigg ( \det \left(  \left([p_1]_\times\right)_{R_1}^{C_1}\right)  \times \\
& \qquad \qquad\qquad \qquad\qquad  \det \left(  \left([p_2]_\times\right)_{R_2}^{C_2} \right)  
 \det \left(\mathcal{A}_C 
 \right) \Bigg) \\
 &= \sum_{i = 1}^3 (-1)^{i+k} p_{3k} p_{3i} \times \\
 &  
\qquad  \sum_{\substack{ |C_1| = |C_2| = 2 \\ C_1 \neq R_1  \\ C_2 \neq R_2 }}  
 \Bigg (
 \det \left( \left([p_1]_\times\right)_{R_1}^{C_1} \right) \times \\
 &\qquad\qquad \qquad\qquad \det \left( \left([p_2]_\times\right)_{R_2}^{C_2} \right)  
 \det \left( \mathcal{A}_{\{C_1, C_2, p_{3i}\}}\right)
 \Bigg) \\
\end{align*}
\begin{align*}
\quad &= (-1)^k p_{3k}  \sum_{i = 1}^3 (-1)^{i} p_{3i} \times\\
&\qquad \qquad\ \sum_{2 \le l,m \le 3 }    
\Bigg (\det\left(  \left([p_1]_\times\right)_{\{p_{12}, p_{13}\}}^{\{p_{11}, p_{1l}\} } \right) \times \\
& \qquad \qquad \qquad\qquad\quad
\det\left( ([p_2]_\times)_{\{p_{22}, p_{23}\}}^{\{p_{21}, p_{2m}\}}\right)  \times \\
& \qquad \qquad \qquad\qquad\quad
\det \left( \mathcal{A}_{\{p_{11}, p_{1l} , p_{21}, p_{2m}, p_{3i}\}}\right )  \Bigg)   \\
&= (-1)^k p_{3k}  \sum_{i = 1}^3 (-1)^{i} p_{3i}\times\\
&\qquad \qquad\quad\sum_{2 \le l,m \le 3 }(-1)^{l+m}  p_{1l}p_{2m} \det \left( \mathcal{A}_{\{p_{11}, p_{1l} , p_{21}, p_{2m}, p_{3i}\}} \right) \\ 
&= (-1)^k p_{3k} \det ( \mathcal{A}(p)_{\{p_{11},p_{21}\}})
\end{align*}
Observe that the final equality follows from expanding the determinant of $\mathcal{A}(p)_{\{p_{11},p_{21}\}}$ on the $p_3$ column.

For general $j_1, j_2$, performing the same computation with $R_1 = \{p_{11},p_{12},p_{13}\} \minus \{p_{1j_1}\}$,  $R_2 = \{p_{21}, p_{22},p_{23}\} \minus \{p_{1j_2}\}$ and $R_3 = \{p_{3k}\}$ yields $\det \left( P(p)_R \mathcal{A} \right) = (-1)^k p_{3k} \det \left( \mathcal{A}(p)_{\{p_{1j_1},p_{2j_2}\}}\right)$. 
\end{proof}
\end{document}